\RequirePackage{tikz}

\documentclass[pdflatex,sn-mathphys]{sn-jnl}

\usepackage{bbm}
\usepackage{subfigure}

\usetikzlibrary{decorations.pathreplacing,calligraphy, shapes.geometric}

\usepackage{mathtools}
\DeclarePairedDelimiter\abs{\lvert}{\rvert}%
\newcommand{\vect}[1]{\boldsymbol{#1}}

\jyear{2023}%

\theoremstyle{thmstyleone}%
\newtheorem{theorem}{Theorem}
\newtheorem{proposition}[theorem]{Proposition}%
\newtheorem{corollary}[theorem]{Corollary}
\newtheorem{lemma}[theorem]{Lemma}

\theoremstyle{thmstyletwo}%
\newtheorem{example}{Example}%
\newtheorem{remark}{Remark}%

\theoremstyle{thmstylethree}%

\begin{document}

\title[Multi-component Matching Queues]{Multi-component Matching Queues in Heavy Traffic}


\author*[1]{\fnm{Bowen} \sur{Xie}}\email{xie.b@wustl.edu,  bowenx@alumni.iastate.edu}




\affil*[1]{\orgdiv{Department of Statistics and Data Science}, \orgname{Washington University in St. Louis}, \orgaddress{\street{1 Brookings Drive}, \city{St. Louis}, \postcode{63130}, \state{Missouri}, \country{US}}}




\abstract{
We consider multi-component matching systems in heavy traffic consisting of $K\geq 2$ distinct perishable components which arrive randomly over time at high speed at the assemble-to-order station, and they wait in their respective queues according to their categories until matched or their ``patience" runs out. 
An instantaneous match occurs if all categories are available, and the matched components leave immediately thereafter. 
For a sequence of such systems parameterized by $n$, we establish an explicit definition for the matching completion process, and when all the arrival rates tend to infinity in concert as $n\to\infty$, we obtain a heavy traffic limit of the appropriately scaled queue lengths under mild assumptions, which is characterized by a coupled stochastic integral equation with a scalar-valued non-linear term. 
We demonstrate some crucial properties for certain coupled equations and exhibit numerical case studies. 
Moreover, we establish an asymptotic Little's law, which reveals the asymptotic relationship between the queue length and its virtual waiting time. 
Motivated by the cost structure of blood bank drives, we formulate an infinite-horizon discounted cost functional and show that the expected value of this cost functional for the $n$th system converges to that of the heavy traffic limiting process as $n$ tends to infinity.  
}




\maketitle




\section{Introduction}\label{sec1}

We consider a queueing model with a matching etiquette that matches multiple categories of components to produce a single product. 
The components of each distinct category arrive sequentially over time and wait in their respective queues. 
To make a final product, we need one component of each category, and once matched, the matched components leave the system immediately. 
The matching philosophy is according to the first-come-first-matched discipline (FCFM). These components could be ``impatient", and they may abandon the system without being matched when their patience runs out. 
Such an assumption is quite natural if the components are perishable or they are of no use after some time. 
Since matching is instantaneous, one can observe that there cannot be all positive numbers of components available throughout all the categories simultaneously at any given time; namely, at least one queue is empty at any given time instant. 
Such queueing models are known as multi-component matching queues with impatient components. Figure \ref{Schematic diagram of matching operation} exhibits a schematic diagram of such a matching operation with three categories of components, where the queue of category $\textcolor{green}{\blacktriangle}$ is empty at this time instant.

\begin{figure}[h!]
    \centering
    \begin{tikzpicture}[scale=0.80, circ/.style={shape=circle, fill, inner sep=2pt, draw, node contents=}]
        \node at (0,2) [rectangle,draw, thick] (product) {\footnotesize Assembly-like matching};
        
        \draw[->, thick] (-1.5, 1) -- (-0.5, 1.5);
        \draw[thick] (-2, 1) rectangle (-1, -1);
        \node at (-1.5, 0.7) [red, rectangle, draw, fill]{}; 
        \node at (-1.5, 0.3) [red, rectangle, draw, fill]{};
        \node at (-1.5, -0.1) [red, rectangle, draw, fill]{};
        \node at (-1.5, -0.5) [red, rectangle, draw, fill]{};
        
        \draw[->, thick] (0, 1) -- (0, 1.5);
        \draw[thick] (-0.5, 1) rectangle (0.5, -1);
        \node at (0, 0.7) [blue, circ, draw, fill, scale=0.7];
        \node at (0, 0.3) [blue, circ, draw, fill, scale=0.7];
        
        \draw[->, thick] (1.5, 1) -- (0.5, 1.5);
        \draw[thick] (1, 1) rectangle (2, -1);
        
        \draw [->, thick] (2.5, 2) -- (6, 0.5);
        \draw node (nodeA) at (2.5, 2) {};
        \draw node (nodeB) at (6, 0.5) {};
        \draw (nodeA) -- (nodeB) node [midway, above, sloped] (TextNode) {\footnotesize Matched triple leaves};
        \draw (nodeA) -- (nodeB) node [midway, below, sloped] (TextNode) {\footnotesize the system immediately};
        
        \draw [decorate, decoration = {calligraphic brace}, thick] (6.5,-1) --  (6.5,2);
        \node at (8,1.5) [rectangle,draw, thick] (product) {\footnotesize Product ($\textcolor{red}{\blacksquare} \textcolor{blue}{\bullet} \textcolor{green}{\blacktriangle}$)};
        \node at (8,0.5) [rectangle,draw, thick] (product) {\footnotesize Product ($\textcolor{red}{\blacksquare} \textcolor{blue}{\bullet} \textcolor{green}{\blacktriangle}$)};
        \draw [dotted, thick] (8, 0) -- (8, -1);

        \draw[white, ultra thick]  (-2.2, -1) -- (2.2, -1);
        
        \draw[->, thick] (-1.5, -1.6) -- (-1.5, -1.2);
        \draw[->, thick] (0, -1.6) -- (0, -1.2);
        \draw[->, thick] (1.5, -1.6) -- (1.5, -1.2);
        \draw node at (0, -2) {New arrivals};
    \end{tikzpicture}
    \caption{Schematic diagram of a matching operation for a product made of components from three distinct categories $\textcolor{red}{\blacksquare}$, $\textcolor{blue}{\bullet}$, $\textcolor{green}{\blacktriangle}$. }
    \label{Schematic diagram of matching operation}
\end{figure}
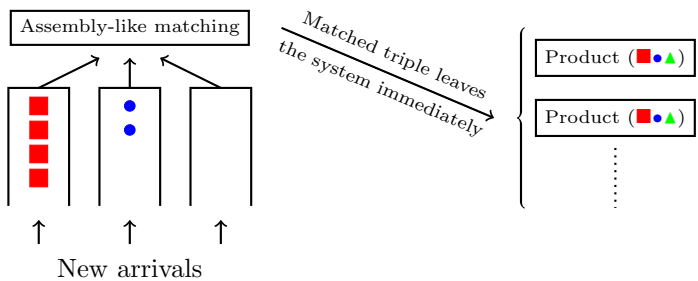

\subsection{Motivation}

Our multi-component matching queue model serves as a generalization of the double-ended queueing system in real life, which is driven by applications in taxi queueing systems \cite{kashyap1966double}, production-inventory systems (cf. 
\cite{kaspi1983inventory}, \cite{perry1999perishable}, \cite{xiegao2023long}, \cite{lee2021optimal}), blood bank drives \cite{bar2017blood}, organ transplantation problems (cf. \cite{boxma2011new}, \cite{khademi2021asymptotically}), and ride-sharing problems \cite{ozkan2020dynamic}, and high-tech manufacturing industries, and etc, where matching occurs between demand (or buyer) and supply (or seller). 
In these settings, one can define a single process as the difference of two queue lengths to represent the queue lengths for both queues by taking its positive part and negative part. This move is normally called a double-ended queue. 
However, it no longer works for many-component matchings. 

In this article, we consider matching queue systems containing $K\geq 2$ distinct categories of components. 
One may interpret this as an interface provided to customers to assemble a device with multiple parts. For example, a personal computer is built upon a motherboard, CPU, Memory, Storage, etc. 
Such a matching etiquette appears in many other pharmaceutical systems, physical systems, and healthcare systems, and it is normally called assemble-to-order (ATO) systems, in contrast to double-ended systems. 
For instance, a pharmaceutical company makes a pharmaceutical product that needs several distinct active pharmaceutical ingredients (APIs). 
Each substance arrives at an ATO production facility at high speed, and each has a short lifetime. 
They await in their respective queues according to their categories once they arrive. 
If any substance is not used before its expiration date, it must be disposed of and removed from the system. 
To produce a pharmaceutical product by establishing an instantaneous match, the assembly station requires one input API of each category. 
Notice that at least one API queue is empty at any given time; otherwise, there will be matches, which instantly yield empty queues due to the immediate matching philosophy. 
To maintain the stability of the system, we assume the substances arrive at the same average speed and are subject to short lifetimes. 
We would like to understand the behavior of this model when the average arrival speed of each API tends to infinity in concert.

To formulate such a multi-component matching queue model, 
let the state process vector represent the queue length of different components. 
It is useful to study such a model since the computation takes more effort for large-scale systems when $K$ is large. Additionally, direct analysis involves significant difficulties when dealing with various states of the queue lengths (See Section \ref{sec: stochastic model}). 
However, its limiting process reveals an appealing structure and provides some insights into its dynamics (see Section \ref{sec: numerical simulations}), and it is easy to simulate for large-scale systems since it can essentially be interpreted by a fixed point theorem under proper space. Moreover, it provides a good approximation when the average arrival speed of components is reasonably large. 
The approximation strategy for queueing models in heavy traffic using Brownian systems is an effective approach from the perspective of quantitative and qualitative insights (cf. \cite{reed2008approximating}, \cite{koccauga2010admission}, \cite{weerasinghe2014diffusion}).

\subsection{Contributions}

A challenge faced in this work lies in tackling the matching completions, which demonstrates the cumulative number of matches that occurred. 
The matching completion enables us to illustrate that the coupling behavior, which is also preserved in the heavy-traffic limiting process, significantly escalates the complexity of the matching queue models. 
This also distinguishes the multi-component matching queues from the double-ended queues (cf. \cite{liu2019diffusion}, \cite{liu2021admission}) since the latter cancels out the matching completions by a coupling behavior.

We summarize the novelty of our work as follows: (i) 
Our matching completion construction is explicit and novel in the sense that it concisely establishes a multi-component matching mechanism incorporating perishable/impatient components, whose twisting behaviors in heavy traffic are carefully analyzed theoretically and numerically. 
It further reveals some crucial properties: stickiness and reflectiveness. 
Its minimum-type scalar-valued structure is also preserved in the heavy-traffic limiting process, and we provide a semimartingale decomposition for a specific matching model, where we observe some underlying local time processes that prevent the queue lengths from dropping below the origin acting as regulators.
(ii) When the arrival rates of those distinct categories of components tend to infinity in concert, we obtain a heavy traffic approximation of the appropriately scaled state process vector under Markovian assumptions in Theorem \ref{weak convergence theorem}, where the heavy traffic limit is characterized by a non-trivial coupled stochastic integral equation. 
Such a coupled stochastic integral equation involves a scalar-valued non-linear term, which also renders entries of the limiting state process vector mutually coupled. 
Figure \ref{Graph of sample paths of matching queues} exhibits a sample instance of the coupling behavior of the heavy traffic limiting process vector $(X_1, X_2, X_3)$ in the case of $K=3$, where at any given time there exists at least one empty entry, and the sample paths also reveal a stickiness. 
\begin{figure}[h!tb]
    \centering
    \includegraphics[scale=0.45]{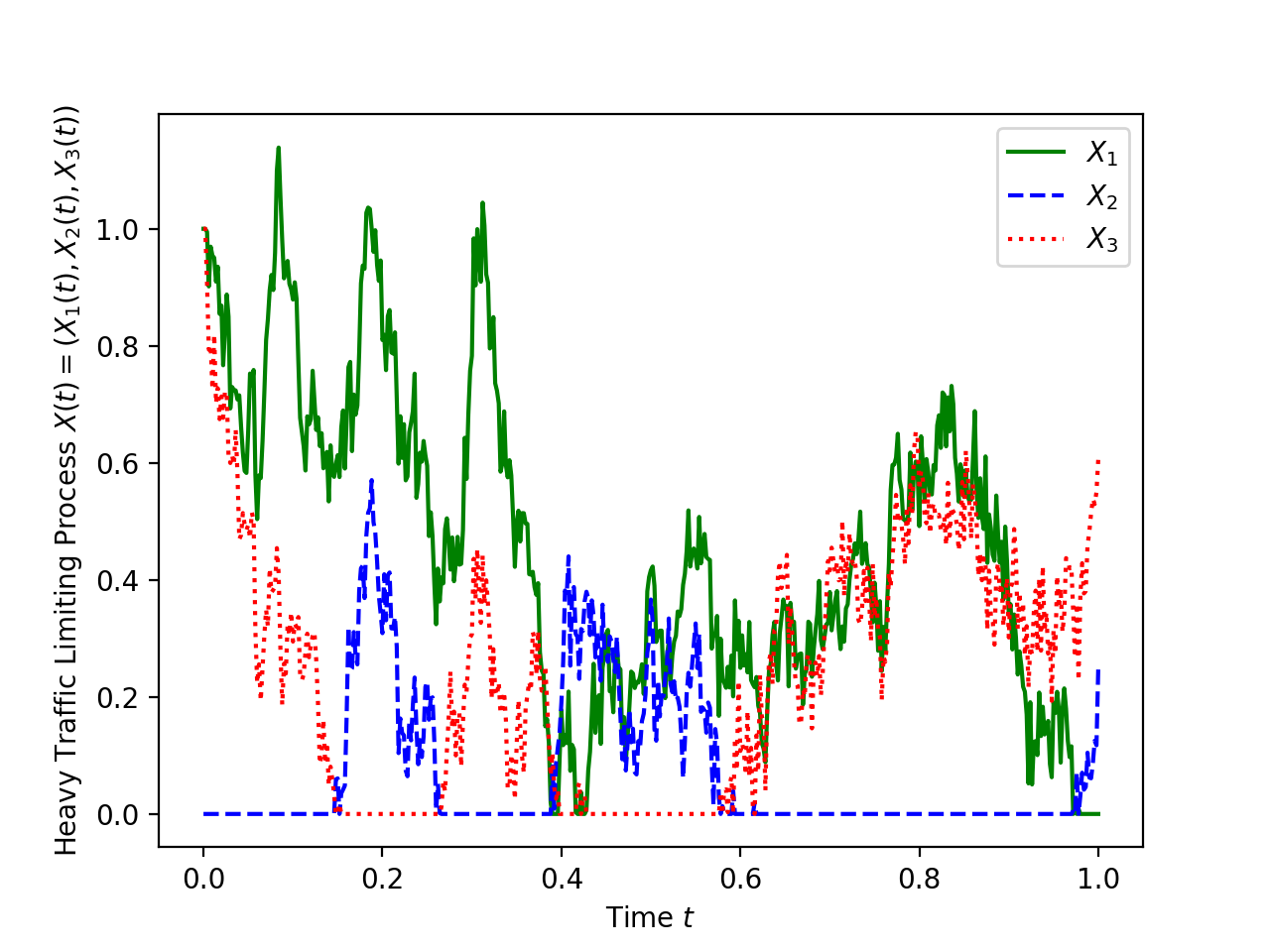}
    \caption{Sample paths of the heavy traffic limit \eqref{limiting process (model with abandonment)} (see Theorem \ref{weak convergence theorem}) in the case of $K=3$}
    \label{Graph of sample paths of matching queues}
\end{figure}
(iii) For each category, we establish an asymptotic relationship between the queue length and its corresponding virtual waiting time, which is often called the asymptotic Little's law (see Theorem \ref{little's law}). We also develop an interesting moment bound result for the virtual waiting time of a specific matching model without reneging under general assumptions (see Proposition \ref{weak convergence theorem (non-abandonment)}) using its exclusive order preserving property (see Proposition \ref{SB of hat V (non-abandonment)}). 
(iv) We also exhibit that the expected value of a properly defined cost functional for the $n$th system converges to that of the heavy traffic limiting process as $n$ tends to infinity (see Theorem \ref{Theorem convergence of cost functionals (polynomial)}), where we admit an unusual restriction of the discount rate related to the number of categories $K$ when considering the infinite-horizon discounted cost functional. 
(v) We provide some insights into coupled queueing systems by performing numerical simulations to illustrate the dynamics of matching completion and abandonment for large category number $K\to\infty$. We also exhibit the variation of some factors that may affect the stickiness in the perspective of the amount of time stays at zero. 

\subsection{Literature Review}

The stochastic matching queue analysis has gained a lot of attention in recent literature \cite{mairesse2020editorial}. 
The double-ended system has been well studied recently, and corresponding control problems have been concerned (cf.
\cite{conolly2002double},  
\cite{liu2015diffusion}, 
\cite{liu2019diffusion},   \cite{liu2021admission},  \cite{lee2021optimal}, etc.). 
In \cite{conolly2002double}, the effect of reneging is studied in the context of double-ended queues, where each demands service from the other to provide a theoretical but brief numerical assessment of operational consequences. 
In \cite{liu2015diffusion}, under a suitable asymptotic regime, they established fluid and diffusion approximations for the queue length process, which are characterized by an ordinary differential equation and time-inhomogeneous asymmetric Ornstein-Uhlenbeck process. They also exhibited the interchangeability of the heavy traffic and steady state limits. 
In \cite{liu2019diffusion}, they studied a double-ended system with two classes of impatient customers and established simple linear asymptotic relationships between the diffusion-scaled queue length process and the diffusion-scaled offered waiting time processes in heavy traffic. 
They also showed that the diffusion-scaled queue length process converges weakly to a diffusion process that admits a unique stationary distribution.
In \cite{lee2021optimal}, they studied a double-ended system having backorders and customer abandonment and determined the optimal (nonstationary) production rate over a finite time horizon to minimize the costs incurred by the system under some cost structure. 

In recent years, there has also been growing interest in two-side matching etiquette on networks/graphs such as bipartite graphs in, for example, \cite{castro2020matching}, \cite{weiss2020directed}, 
\cite{kohlenberg2023cost}
, etc. 
In particular, \cite{castro2020matching} studied a parallel matching queue with reneging under an FCFM manner according to a compatibility graph and provided product forms of steady-state distributions, and \cite{weiss2020directed} considered a two-sided parallel service model consisting of agents of several types and goods of several types, with a bipartite compatibility graph between agent and good types under an FCFM manner in the views of matching rates and delays, whose matching principles are not multi-component matching. 
\cite{kohlenberg2023cost} studied two-sided matching queues with abandonment, where they identified non-asymptotic and universal scaling laws for the matching loss (cost-of-impatience)
and established operating regimes that arise in asymptotic approximations. They characterized the trade-off between abandonment and capacity costs of the two-sided matching model compared with the single-sided queue. 
One may observe that our matching model can be simplified to a double-ended queue when $K=2$ and we also include asymptotic results for cost functionals; however, more cost-of-impatience characterizations under different operating regimes and its control problem of optimal capacity scaling may also be of general interest for multi-component matching, which will be considered in our follow-up paper \cite{xiewu2023control} and future studies. 

The generalized multi-class matching has also been studied in the literature with different formulations through various aspects (cf. \cite{harrison1973assembly}, \cite{plambeck2006optimal}, \cite{gurvich2015dynamic}, etc.). 
In \cite{harrison1973assembly}, Harrison studied a model with an assembly-like behavior to produce a product with several components and developed limit theorems for the appropriately normalized versions of the associated vector waiting time process in heavy traffic. 
The model contains $K\geq2$ independent renewal input processes. 
The server requires one input component of each category $j = 1, \cdots, K$, and once the server has all the required components, it takes a random processing time to finish the product.
\cite{plambeck2006optimal} introduced a model with order queues and component queues and studied a control problem of an ATO system with a high volume of prospective customers arriving per unit time, where multiple different components are instantaneously assembled into different finished products, and the control problem is developed so that they can maximize the expected infinite-horizon discounted profit by choosing product prices, component production capacities, and a dynamic policy for sequencing customer orders for assembly. 
In \cite{gurvich2015dynamic}, the authors studied a matching system with instantaneous processing, where a system manager could control which matchings to execute given multiple options and addressed the problem of minimizing finite-horizon cumulative holding costs. 
They established a multi-dimensional imbalance process to characterize the matching model and devised a myopic discrete-review matching control, which is shown to be asymptotically optimal in heavy traffic. 

Our model differs from \cite{plambeck2006optimal} since they introduced additional order queues to match with product queues and defined a shortage process to characterize the difference between the number of components required to assemble all outstanding orders and the number of components in inventory. 
In their asymptotic result, they established that the order queue lengths and inventory levels (component queue lengths) viewed under diffusion scaling are approximately deterministic functions of the shortage process in high volume. 
By considering a single infinity demand in \cite{plambeck2006optimal} or a single product-production system in \cite{gurvich2015dynamic}, our model might be thought of as a special case of theirs if we discard perishable assumption. 
However, to the best of our understanding and in the process-level view, we incorporate perishable components so that the limiting dynamics of our model depend on the current state in addition to the underlying Brownian motions, which makes those two results not applicable. 
We explicitly establish the matching completion process using inputs (arrivals and abandonments). 
Moreover, our formulation is not an immediate consequence of their asymptotic results in heavy traffic even without perishable assumption, and one needs more non-trivial work to rigorously exhibit the equivalence of our heavy-traffic limits and theirs (see Remark \ref{remark: compare with literature}). 

It is also worth mentioning that the multi-component matching models (without reneging) may be formulated as a particular matching system with hypergraphical matching structures that have a single hyperedge with $K\geq 2$ nodes, see for instance \cite{rahme2021stochastic}, where they studied the stability of various hypergraph geometries and demonstrated that the stochastic matching models on hypergraphs are, in general, difficult to stabilize. 
Other relevant graphical or hypergraphical matching model interpretations can be found in  \cite{buke2015stabilizing}, \cite{mairesse2016stability}, \cite{nazari2019reward}, 
\cite{jonckheere2023generalized}. 

More results of the matching systems can be found in \cite{green1985queueing}, \cite{adan2009exact}, \cite{adan2018reversibility}, \cite{fazel2018approximating}, etc. 
An application of a ride-sharing system is studied in \cite{ozkan2020dynamic}, and an organ transplant system is studied in \cite{khademi2021asymptotically}.

\subsection{Organization}

The rest of this article is organized as follows: 
In Section \ref{sec: stochastic model}, we introduce the stochastic model along with its assumptions and the heavy traffic conditions. 
Section \ref{sec: weak convergence} is mainly devoted to the heavy traffic limit of the diffusion-scaled queue lengths, which is characterized by a coupled multivariate stochastic integral equation in Theorem \ref{weak convergence theorem}. 
We also exhibit the semimartingale decomposition of a special coupling equation to reveal the regulated property through underlying local time processes. 
In Section \ref{sec: little's law}, an asymptotic relationship between the queue length and its virtual waiting time is formulated in Theorem \ref{little's law}, and an interesting moment bound result (see Proposition \ref{SB of hat V (non-abandonment)}) of a proper scaled waiting time is established using a novel proof, where we utilize an order-preserving property for a specific matching queue model. 
In Section \ref{sec: convergence of cost functionals}, we establish the convergence of the expected value of an infinite-horizon discounted cost functional of queueing systems to that of the heavy traffic limits under some constraint over the discount factor. 
Section \ref{sec: numerical simulations} is devoted to numerical simulations of case studies. 
We conclude in Section \ref{sec: conclusion}. 
Appendix \ref{appendix1} contains all of our proofs except the proofs of our main weak convergence result. 

\textbf{Notation.} Let $\mathbb{N}$ represent the set of positive integers. Let $\mathbb{R}$ denote the one dimensional Euclidean space and $\mathbb{R}^K = \mathbb{R}\times\cdots\times\mathbb{R}$ denote the product of $K$ of Euclidean space $\mathbb{R}$. 
For $0< T\leq \infty$, let $D[0, T]$ denote the Skorokhod space of functions with right continuous and left limits (RCLL) and let $D^K[0, T]$ denote the product of $K$ of Skorokhod space $D[0, T]$. 
The vector norm and Frobenius matrix norm are defined by 
$\|\vect{x}\| = (\sum_{j=1}^K \abs{x_j}^2)^{\frac{1}{2}}$ and $\|\vect{y}\| = (\sum_{i=1}^K \sum_{j=1}^K \abs{y_{ij}}^2)^{\frac{1}{2}}$
for $\vect{x}\in\mathbb{R}^{K\time1}$ and $\vect{y}\in\mathbb{R}^{K\times K}$. 
Let $\vect{I} = (1, \cdots, 1)^\intercal\in\mathbb{R}^K$ denote a constant one vector and $\diamond$ represent the Hadamard entrywise product. 
The uniform norm on $[0, T]$ for process $\vect{X}$ in $D^{K}[0, T]$ is defined by
\begin{equation*}
    \|\vect{X}\|_T = \sup_{t\in[0, T]}\|\vect{X}(t)\|.  
\end{equation*}
Throughout, we use $\Rightarrow$ to denote weak convergence in $D^K[0, T]$. 
For any real number $a$, $a^+ = \max\{a, 0\}$ and $a^- = \max\{-a, 0\}$. For any two real numbers $a$ and $b$, $a\wedge b = \min\{a, b\}$ and $a\vee b = \max\{a, b\}$. 


\section{Stochastic Model}
\label{sec: stochastic model}

We consider a multi-component matching system on a probability space $(\Omega, \mathcal{F}, P)$. 
We assume a product is made of $K\geq 2$ disparate perishable components. 
Each component arrives randomly over time at the assembling station and waits in their respective queues according to their categories until matched. 
If a component of each category is available, then matching occurs instantaneously following FCFM discipline, and after that, those matched items leave the system instantaneously. 
In this model, each component departs from the system for two reasons: either getting matched or losing patience. 
Figure \ref{sample matching graph} shows a simple example of this matching system: each symbol represents a component, colors/shapes represent their different categories, and $\lambda_i$'s and $\delta_i$'s represent the arrival rates and abandonment rates, respectively. 

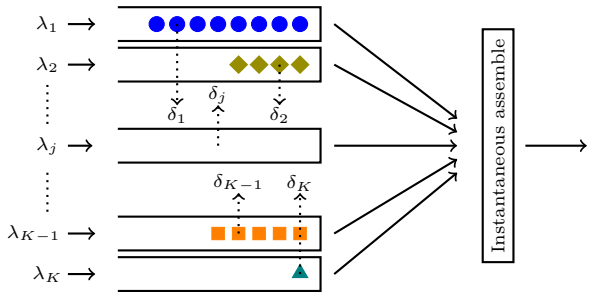
\begin{figure}[h!]
    \centering
    \begin{tikzpicture}[scale=0.90, circ/.style={shape=circle, fill, inner sep=2pt, draw, node contents=}]
        \draw node at (-4, 2) {\footnotesize $\lambda_1$};
        \draw[->, thick] (-3.7, 2) -- (-3.3, 2);
        \draw node at (-4, 1.4) {\footnotesize $\lambda_2$};
        \draw[->, thick] (-3.7, 1.4) -- (-3.3, 1.4);
        \draw [dotted, thick] (-4, 1.1) -- (-4, 0.5);
        \draw node at (-4, 0.2) {\footnotesize $\lambda_j$};
        \draw[->, thick] (-3.7, 0.2) -- (-3.34, 0.2);
        \draw [dotted, thick] (-4, -0.2) -- (-4, -0.8);
        \draw node at (-4.2, -1.1) {\footnotesize $\lambda_{K-1}$};
        \draw[->, thick] (-3.7, -1.1) -- (-3.3, -1.1);
        \draw node at (-4, -1.7) {\footnotesize $\lambda_K$};
        \draw[->, thick] (-3.7, -1.7) -- (-3.3, -1.7);
        
        \draw[thick] (-3, 2.25) rectangle (0, 1.75);
        \draw[thick] (-3, 1.65) rectangle (0, 1.15);
        \draw[thick] (-3, 0.45) rectangle (0, -0.05);
        \draw[thick] (-3, -0.85) rectangle (0, -1.35);
        \draw[thick] (-3, -1.45) rectangle (0, -1.95);

        \draw[white, ultra thick]  (-3, 2.3) -- (-3, -2);
        \draw node (q11) at (-0.3, 2) [blue, circ];
        \draw node (q12) at (-0.6, 2) [blue, circ];
        \draw node (q13) at (-0.9, 2) [blue, circ];
        \draw node (q14) at (-1.2, 2) [blue, circ];
        \draw node (q15) at (-1.5, 2) [blue, circ];
        \draw node (q16) at (-1.8, 2) [blue, circ];
        \draw node (q17) at (-2.1, 2) [blue, circ];
        \draw node (q18) at (-2.4, 2) [blue, circ];
        

        \node at (-0.3, 1.4) [olive, diamond, draw, fill, scale=0.5]{}; 
        \node at (-0.6, 1.4) [olive, diamond, draw, fill, scale=0.5]{}; 
        \node at (-0.9, 1.4) [olive, diamond, draw, fill, scale=0.5]{}; 
        \node at (-1.2, 1.4) [olive, diamond, draw, fill, scale=0.5]{}; 
        
        \node at (-0.3, -1.1) [orange, rectangle, draw, fill, scale=0.7]{}; 
        \node at (-0.6, -1.1) [orange, rectangle, draw, fill, scale=0.7]{}; 
        \node at (-0.9, -1.1) [orange, rectangle, draw, fill, scale=0.7]{}; 
        \node at (-1.2, -1.1) [orange, rectangle, draw, fill, scale=0.7]{}; 
        \node at (-1.5, -1.1) [orange, rectangle, draw, fill, scale=0.7]{};
        
        \node at (-0.3, -1.7) [teal, isosceles triangle, isosceles triangle apex angle=60, rotate=90, draw, fill, scale=0.4]{};

        \draw [->, thick] (0.2, 2) -- (2, 0.6);
        \draw [->, thick] (0.2, 1.4) -- (2, 0.4);
        \draw [->, thick] (0.2, 0.2) -- (2, 0.2);
        \draw [->, thick] (0.2, -1.1) -- (2, 0);
        \draw [->, thick] (0.2, -1.7) -- (2, -0.2);

        \node at (2.6, 0.2) [rectangle, draw, thick, rotate=90] (product) {\footnotesize Instantaneous assemble};
        \draw [->, thick] (3, 0.2) -- (3.9, 0.2);

        \draw [->, thick, dotted] (-2.1, 2) -- (-2.1, 0.8);
        \draw node at (-2.1, 0.65) {\footnotesize $\delta_1$};
        \draw [->, thick, dotted] (-0.6, 1.4) -- (-0.6, 0.8);
        \draw node at (-0.6, 0.65) {\footnotesize $\delta_2$};
        \draw [->, thick, dotted] (-1.5, 0.2) -- (-1.5, 0.8);
        \draw node at (-1.5, 0.95) {\footnotesize $\delta_j$};
        \draw [->, thick, dotted] (-1.2, -1.1) -- (-1.2, -0.5);
        \draw node at (-1.2, -0.35) {\footnotesize $\delta_{K-1}$};
        \draw [->, thick, dotted] (-0.3, -1.7) -- (-0.3, -0.5);
        \draw node at (-0.3, -0.35) {\footnotesize $\delta_K$};

    \end{tikzpicture}
    \caption{A queueing network view of a matching model with $K$ categories of components}
    \label{sample matching graph}
\end{figure}

Here, we are interested in a fast system as all the components 
arrive at high speed in concert and are subject to an identical average arrival rate. 
Moreover, each component has a short lifetime. 
The complexity of this model comes from the occurrence of abandoned components. 
Since the matching is instantaneous, we observe that it is not possible to have all positive numbers of components of each category waiting in their queues simultaneously. Thus, at least one queue corresponding to one of the categories is empty at any given time. 
In some scenarios, if we have two empty queues at some given time, the rest of the components need to wait for the missing parts. When one of the missing components arrives, it still has to wait for the other missing component. However, this component may abandon the system before the arrival of the last missing component, which again generates an empty queue. 
This phenomenon significantly makes the model more complicated than the one with no abandonment, and the matching queue model with abandonment is more realistic in the real world. 

In the rest of this section, we intend to exhibit the difficulty of direct analysis and then introduce the asymptotic framework of the matching queue system along with some basic assumptions. 

\subsection{Generator of the Matching Queue}

It is natural to perform direct analysis under Markovian assumptions such that arrivals follow Poisson distributions with arrival rate $\lambda_i>0$ and patience times $d_{i, k}$'s of each component $k$ in category $i$ are exponential distributed with rate $\delta_i>0$ for each category $i\in\{1, \cdots, K\}$. 
Here, the patience time of a component is independent of its arrival time as well as the arrival times and patience times of those components who arrived earlier, and they are also independent of everything else in the system. 
However, one may expect severe difficulties due to the coupling behavior of the queue length processes, and therefore, the asymptotic analysis comes into play (see Section \ref{sec: asymptotic framework}), which will be discussed in great detail in later sections. 

For 
completeness
and to take a closer look, we present a construction of the generator of the queue lengths. 
One can observe that the occurrence of a match depends on all distinct categories. 
The state process should properly reflect such a relationship: at least one queue is empty at any time. Hence, we define the state space as 
\begin{equation}
E := \left\{(s_1, s_2, \cdots, s_K)\in E^{(1)}\times E^{(2)}\cdots\times E^{(K)}: \prod_{j=1}^K s_j = 0 \right\}, 
\label{eq: state space}
\end{equation}
where $E^{(j)} := \{0, 1, \cdots\}$ for $j = 1, \cdots, K$, and $s_j$'s denote the queue length of the $j$th queue. 

Under the Markovian assumptions, the queue length vector $(Q_1, \cdots, Q_K)$ is a Markov chain on $\mathbb{Z}_+^K$ with rate matrix given by
\begin{equation}
\begin{aligned}
&\quad Q((s_1, \cdots, s_K), (s_1', \cdots, s_K')) \\
&= 
\begin{cases}
\lambda_i, & \text{ if }\left\{s_1' = s_1, \cdots, s_i' = s_i + 1, \cdots, s_K' = s_K, \text{ and } \prod_{j\neq i}^K s_j = 0\right\}, \\
& \text{ or } \left\{s_1' = s_1 - 1, \cdots, s_i' = s_i, \cdots, s_K' = s_K - 1, \text{ and } \prod_{j\neq i}^K s_j \neq 0\right\}, \\
s_i\delta_i, & \text{ if } s_1' = s_1, \cdots, s_i' = s_i - 1, \cdots, s_K' = s_K,  
\end{cases}
\end{aligned}
\label{eq: rate matrix}
\end{equation}
where $i = 1, 2, \cdots, K$, and $(s_1, \cdots, s_K), (s_1', \cdots, s_K')\in E$. 
To understand the rate matrix, we may take the queue length of the first queue as an example, where we have $i=1$ in \eqref{eq: rate matrix}. Since we do not know the queue structure at this time instant, we have to decompose our state into two cases: first, the other queues must have at least one empty queue; second, none of the other queues are empty, which is also the complementary event of the former case. In the former case, a new arrival to the first queue leads to an increment of the queue length of the first queue and cannot formulate any matches since some queues (other than the first queue) are empty. However, in the latter case, the first queue must be an empty queue due to the state space \eqref{eq: state space}, and a new arrival to the first queue results in a match. Further, components of category $i$ may abandon their queue with abandonment rate $s_i\delta_i\geq 0$. 

Therefore, the generator for the pure jump process $(Q_1, \cdots, Q_K)$ can be written as
\begin{equation}
\begin{aligned}
    &\quad Af(s_1, \cdots, s_K) \\
    &= \sum_{i=1}^K \lambda_i \bigg[\left(f(s_1, \cdots, s_i+1, \cdots, s_K) - f(s_1, \cdots, s_K)\right) \mathbbm{1}_{\left[\prod_{j\neq i}^K s_j = 0\right]}\\
    &\quad\quad\quad\quad\quad +(f(s_1 - 1, \cdots, s_i, \cdots, s_K-1) - f(s_1, \cdots, s_K)) \mathbbm{1}_{\left[\prod_{j\neq i}^K s_j \neq 0\right]}\bigg]\\
    &\quad + \sum_{i=1}^K s_i\delta_i (f(s_1, \cdots, s_i - 1, \cdots, s_K) - f(s_1, \cdots, s_K)),  
\end{aligned}
\label{eq: generator}
\end{equation}
where $f\in C^2(\mathbb{R}^K)$. 
One can observe that its direct analysis may not be trivial due to the coupling behaviors, which are characterized by the indicator functions.

\subsection{Asymptotic Framework}
\label{sec: asymptotic framework}

To perform asymptotic analysis, we develop a sequence of independent matching queue systems parameterized by $n\in\mathbb{N}$ such that the arrival rate of each queue gets increasingly large without bound in concert when we let $n$ tend to infinity. Quantities that depend on $n$ have $n$ as a superscript in their notations, and a subscript tells the associated category. Since the matching happens instantaneously, we can interpret it as an extremely large service rate in normal queueing systems. 
Intuitively, suppose we speed up the whole system by letting $n\to\infty$. 
The arrival rates get extremely large to obtain heavy traffic conditions, which leads to a situation where even though components from each category arrive quite frequently, intuitively, the instantaneous matching and abandonment should guarantee non-explosive queues. 
On the other hand, if the inter-arrival times are large, instantaneous matching may lead to numerous empty queues since they may not be patient enough. 

Within these facts, we construct the $n$th matching queue system. 
Let the queue length process, $\vect{Q^n}(\cdot) = (Q_{1}^{n}(\cdot), \cdots, Q_K^n(\cdot))^\intercal$ denote the state process. 
For each $i\in\{1, \cdots, K\}$, let $A_{i}^{n}(\cdot)$ and $G_i^n(\cdot)$ be two independent processes represent the number of arrivals and abandonments of category $i$ in the $n$th system respectively. 
We assume that $A_i^n(\cdot)$ follows a Poisson process in $D([0, \infty), \mathbb{R})$ with arrival rate $\lambda_i^n > 0$, and $\{A_k^n\}_{1\leq k\leq K}$ are all independent with each other. 
Moreover, we assume $\lambda_{i}^{n}\to\infty$ as $n\to\infty$ for each $i$.
We also assume that the abandonment processes follow independent Poisson processes with respective parameter $\delta_i^n > 0$ such that it is constructed by
\begin{equation}
    G_{i}^{n}(t) := N_i\left(\delta_{i}^{n}\int_0^t Q_{i}^{n}(s)ds\right), 
    \label{abandonment process}
\end{equation}
where $\delta_{i}^{n}>0$ is a constant and $N_i$'s are independent unit rate Poisson processes. 
We assume $\lim_{n\to\infty}\delta_i^n = \delta_i$, where $\delta_i>0$ is a real number. 
More precisely, one can think of the patience time of a component is independent of its arrival time as well as the arrival times and patience times of those components who arrived earlier, and they are also independent of everything else in the system. 
Notice that a random time change (see II.6 in \cite{bremaud1981point} and Chapter 6 in \cite{ethier2009markov}) is employed in the construction above since the instantaneous overall abandonment rate at time $s$ is $\delta_i^nQ_i^n(s)$, which is the multiplication of the number $Q_i^n(s)$ of components waited in queue and the individual patience rate $\delta_i^n$ (see Section 2.1 and 7.1 in \cite{pang2007martingale}). 

Since the occurrence of a match is instantaneous and relies only on the number of arrivals and abandonments, the number of completed matches by time $t$ depends on all the arrivals $(A_1^n(t), A_2^n(t), \cdots, A_K^n(t))$ and the abandonments $(G_1^n(t), G_2^n(t), \cdots, G_K^n(t))$ by time $t$. 
We introduce the natural filtration $\mathcal{F}^n = (\mathcal{F}_t^n)_{t\geq0}$ by
\begin{equation}
    \mathcal{F}_t^n := \sigma\left(Q_i^n(0), A_i^n(s), G_i^n(s): 0\leq s\leq t \text{ and } 1\leq i \leq K\right).
    \label{natural filtration}
\end{equation}
It also represents all the information available regarding the $n$th system at time $t$. 

We describe other basic assumptions and exhibit the heavy traffic assumption for the sequence of matching queue systems as follows. 

\textbf{Assumption 1} (Initial conditions). For each $i\in\{1, \cdots, K\}$, let $Q_i^n(0)\geq0$ denote the number of initial components of category $i$ in the $n$th system. It is assumed to be deterministic and independent of each other and satisfies
\begin{equation}
    \lim_{n\to\infty} \frac{Q_i^n(0)}{\sqrt{n}} = x_i, 
    \label{initial state limit}
\end{equation}
where $x_i\geq0$ is a real number. For convenience, we assume those initial components of each category $i$ do not abandon, and they will get matched eventually. 
This is not a restrictive assumption but is for ease of analysis and can actually be 
relaxed
(see Assumption 2.1 of \cite{liu2021admission}, Lemma 4.1 of \cite{weerasinghe2014diffusion}, Lemma 2.1 of \cite{mandelbaum2012queues}). 

Notice that since the instantaneous matching policy, at least one of the entry in $\vect{Q^n}(0) = (Q_1^n(0), Q_2^n(0), \cdots, Q_K^n(0))^\intercal$ is zero and so does the limiting initial states $\vect{x} = (x_1, \cdots, x_K)^\intercal$. Here we have $\prod_{j=1}^K Q_j^n(0) = \prod_{j=1}^K x_j = 0$. 

\textbf{Assumption 2} (Heavy-traffic condition). 
For each $i\in\{1, \cdots, K\}$, there exists a constant $\lambda_0 > 0$ so that
\begin{equation}
    \lim_{n\to\infty}\frac{\lambda_i^n -\lambda_0 n}{\sqrt{n}}=\beta_i, 
    \label{regime}
\end{equation}
for each $i\in\{1, \cdots, K\}$, where $\beta_i$ is a real number. 

\begin{remark}
    Even though it is natural to assume renewal arrivals and general distributed patience times, it is a challenging problem due to the nature of the instantaneous multiclass matching discipline. Later on in the main Theorem \ref{weak convergence theorem} below, one could see the benefits of the Markovian assumptions that provide a non-trivial coupling martingale representation. However, this is not inherited under general assumptions, and this difficulty also results in uncertain stochastic boundedness of the queue lengths. 
    The model with general assumptions will be addressed in future projects. 
\end{remark}

Under the above assumptions, we introduce the matching discipline, which is characterized by matching completions. Let $\Tilde{R}^n(\cdot)$ represent the cumulative number of matches happened by time $t$ and it is given by
\begin{equation}
    \Tilde{R}^n(t) := \min_{1\leq j\leq K} \left\{Q_j^n(0) + A_j^n(t) - L_j^n(t)\right\}, 
    \label{number of matches with L}
\end{equation}
where the process $L_j^n(\cdot)$ denotes the number of components who entered the $j$th queue by time $t$ and eventually abandoned the system, albeit we do not observe future information. 
Recall that $G_i^n(t)$ introduced in \eqref{abandonment process} counts the number of abandoned components of category $i$ during $[0, t]$. 
Some components of category $i$ in its queue may still abandon after time $t$ and those components will never get matched. 
Hence the matching completions depend entirely on those non-abandoned parts. 
Thus, $L_i^n(\cdot)$ comes into the picture in \eqref{number of matches with L}. 
However, since it is not possible to observe future information at any given time $t$, we need an alternative definition.  
Analogous to \eqref{number of matches with L}, we define 
\begin{equation}
    R^n(t) := \min_{1\leq j\leq K} \left\{Q_j^n(0) + A_j^n(t) - G_j^n(t)\right\},
    \label{number of matches with G}
\end{equation}
for $t\geq0$. 
One can observe that for all $t\geq 0$ and any $i$, 
\begin{equation*}
    R^n(t) -\Tilde{R}^n(t) = \min_{1\leq j\leq K} \left\{Q_j^n(0) + A_j^n(t) - G_j^n(t)-\Tilde{R}^n(t)\right\}, 
\end{equation*}
which equates to $\min_{1\leq j\leq K} \left\{Q_j^n(t)\right\}$ due to the matching discipline. Since at least one queue is empty throughout the evolution, we deduce that $R^n(t) = \Tilde{R}^n(t)$ for all $t\geq 0$, and we can interpret $R^n(t)$ as the number of completed matches by time $t$. 
Thus, a sequence of queue length process of category $i$ is defined as
\begin{equation}
    Q_i^n(t) = Q_i^n(0) + A_i^n(t) - G_i^n(t) - R^n(t).  
    \label{queue length}
\end{equation}
Since $R^n(\cdot)$ process depends on all the categories, we observe that the queue length processes as in \eqref{queue length} for each $i$ are mutually coupled if we manage to cancel out the common $R^n(\cdot)$ process. 
Thus, we can simply call \eqref{queue length} the coupled queue length processes. 
Our objective is to understand the behaviors of an appropriately scaled queue length process when all components arrive quite fast in concert as $n$ tends to infinity.

\section{Weak Convergence}\label{sec: weak convergence}

This section is devoted to asymptotic analysis by characterizing the weak convergence of diffusion-scaled queue lengths in $D^K[0, T]$ as $n\to\infty$. 
First, we introduce the following centered and scaled quantities:
\begin{equation}
\begin{aligned}
    \hat{Q}_i^n(t) &:= \frac{Q_i^n(t)}{\sqrt{n}}, \quad \hat{A}_i^n(t) := \frac{A_i^n(t) - \lambda_i^nt}{\sqrt{n}}, \\
    \hat{G}_i^n(t) &:= \frac{G_i^n(t)}{\sqrt{n}}, \quad \hat{R}^n(t) := \frac{R^n(t) - \lambda_0 nt}{\sqrt{n}},  
\end{aligned}
\label{diffusion scales}
\end{equation}
for all $t\geq0$ and $i\in\{1, \cdots, K\}$. 
By using \eqref{queue length} and \eqref{diffusion scales}, the diffusion-scaled queue length process can be reformulated as
\begin{equation}
    \hat{Q}_i^n(t) = \hat{Q}_i^n(0) + \hat{A}_i^n(t) + \frac{\lambda_i^n - \lambda_0 n}{\sqrt{n}} t - \hat{G}_i^n(t) - \hat{R}^n(t), 
    \label{diffusion scaled queue length}
\end{equation}
where 
\begin{equation}
    \hat{R}^n(t) = \min_{1\leq j\leq K}\left\{\hat{Q}_j^n(0) + \hat{A}_j^n(t) + \frac{\lambda_j^n - \lambda_0 n}{\sqrt{n}} t - \hat{G}_i^n(t)\right\}. 
\end{equation}

\begin{remark}
    Under our assumption of the Markovian arrivals, the diffusion scaled arrival process $\hat{A}_i^n$ satisfies that for each $i$ and $T>0$, 
    \begin{equation}
        \hat{A}_i^n \Rightarrow \sigma_i W_i, 
        \label{weak convergence of hat A}
    \end{equation}
    in $ D[0, T]$ as $n\to\infty$, where $\sigma_i>0$ is a constant and $W_i(\cdot)$'s are $K$ independent standard Brownian motions. It also satisfies the moment condition:
    \begin{equation}
        E\left[\sum_{j=1}^K \|\hat{A}_j^n\|_T^2\right]\leq C_0(1+T^m),
        \label{moment condition on hat A}
    \end{equation}
    for $T>0$, where $C_0$ and $m \geq 1$ are constants independent of $T$ and $n$, and more precisely, $m=1$ for the second-moment case (for details, refer to Lemma 2 in \cite{atar2004scheduling} and Theorem 4 in \cite{krichagina1992diffusion}). 
\end{remark}

We first present the main result of the diffusion approximation of the matching model in the following Theorem \ref{weak convergence theorem}, and the rest of this section will be devoted to its discussion and proof. 

\begin{theorem}
    \label{weak convergence theorem}
    Let $T>0$ and Assumptions 1-2 hold under the above Markovian assumptions. Consider the state process $\vect{\hat{Q}^n}(t) = (\hat{Q}_1^n(t), \cdots, \hat{Q}_K^n(t))^\intercal\in D^K[0, T]$, where $\hat{Q}_i^n(t)$ satisfies \eqref{diffusion scaled queue length} for all $t\geq0$. Then the sequence $(\vect{\hat{Q}^n})$ converges weakly to a diffusion process $\vect{X}$ in the space $ D^K[0, T]$ as $n\to\infty$. Moreover, the heavy traffic limiting diffusion process $\vect{X}(\cdot)=(X_1(\cdot), \cdots, X_K(\cdot))^\intercal$ is a unique strong solution to the coupled stochastic integral equation: 
    \begin{equation}
        \vect{X}(t) = \vect{x} + 
        \vect{\beta} t + 
        \Sigma \vect{W}(t) - \int_0^t \vect{\delta} \diamond \vect{X}(s)ds - R(t)\vect{I},
        \label{limiting process (model with abandonment)}
    \end{equation}
    where $\vect{x}=(x_1, \cdots, x_K)^\intercal$ and each entry $x_i\geq 0$ is a real number given by \eqref{initial state limit}, $\vect{\beta} = (\beta_1, \cdots, \beta_K)^\intercal$ and $\vect{\Sigma} = \text{diag}(\sigma_1, \cdots, \sigma_K)$ are constants given by \eqref{regime} and \eqref{weak convergence of hat A}, $\vect{\delta} = (\delta_1, \cdots, \delta_K)^\intercal$ are positive real numbers given by \eqref{abandonment process}, $\vect{W} = (W_1, \cdots, W_K)^\intercal$ and $\{W_j\}_{1\leq j\leq K}$ are $K$ independent standard Brownian motions, 
    \begin{equation}
        R(t) := \min_{1\leq j\leq K}\left\{x_j +\beta_jt + \sigma_jW_j(t) - \int_0^t \delta_j X_j(s)ds\right\}
        \label{R in limiting process}
    \end{equation}
    for $t\in [0, T]$, and $\vect{I} = (1, \cdots, 1)^\intercal\in\mathbb{R}^K$. 
    Additionally, the product $\prod_{j=1}^K X_j(\cdot) = 0$. 
\end{theorem}

Some comments about Theorem \ref{weak convergence theorem} are in order. First, the joint convergence of $\vect{\hat{Q}^n}$ is critical since the entries of the diffusion-scaled queue length process vector $\vect{\hat{Q}^n}$ are coupled with each other, and this coupling phenomenon is preserved in the heavy traffic limiting process. This happens because of the scalar-valued $\hat{R}^n(\cdot)$ term, which remains identical throughout all the queue length expressions as in \eqref{diffusion scaled queue length}. The same is true for the scalar-valued $R(\cdot)$ process as in \eqref{limiting process (model with abandonment)}, which can be canceled out by substituting it from individual expressions of $X_i$'s. 
Second, the coupling behavior persists in the heavy traffic limit \eqref{limiting process (model with abandonment)}. Thus, we can call \eqref{limiting process (model with abandonment)} a coupled stochastic integral equation. To the best of our knowledge, the heavy traffic limit obtained in Theorem \ref{weak convergence theorem} is quite different in its character from the regular heavy traffic limits of queueing systems. 
We have seen that in Figure \ref{Graph of sample paths of matching queues}, the coupling behavior of the heavy traffic limit \eqref{limiting process (model with abandonment)} remains, which results in zero entries.  
Third, this scalar-valued non-linear term is more complicated than it looks. 
We conjecture that there may be some underlying local time terms that admit a semimartingale decomposition. 
However, such a property is not trivial in general due to the past dependence. 
Particularly, Section \ref{sec: Semimartingale Property of the Coupled Process} provides an 
explicit
semimartingale decomposition for a special case of coupled equations as an instance.

Additionally, in the case of matching queue with no abandonment, we can relax our assumption for the arrival process by assuming renewal-type of arrival processes, and assumptions for initial queue lengths and the heavy traffic assumption are preserved. 
We can obtain an analogous heavy-traffic approximation for this specific model (see \cite{xie2022topics} for more details), which involves a special order-preserving property of its matching discipline (see Section \ref{sec: little's law}). This result will be used to demonstrate an interesting moment-bound result in Proposition \ref{SB of hat V (non-abandonment)} in later discussions, and it could also serve as a motivation for general assumptions in future studies. 
For completion, we exhibit the asymptotic result below, whose proof is omitted. 

\begin{proposition}[Theorem 4 in \cite{xie2022topics}]
    \label{weak convergence theorem (non-abandonment)}
    Let $T>0$ and we assume the arrival process $A_i^n(\cdot)$ follows the independent renewal-type process in $D([0, \infty), \mathbb{R})$ with rate $\lambda_i^n>0$ for each $i$. 
    Suppose the Assumptions 1-2 hold. 
    Then, the queue length $\vect{\hat{Q}^n} = (\hat{Q}_1^n, \cdots, \hat{Q}_K^n)^\intercal$, where 
    \[
    \hat{Q}_i^n(t) = \hat{Q}_i^n(0) + \hat{A}_i^n(t) + \frac{\lambda_i^n - \lambda_0 n}{\sqrt{n}} t - \hat{R}^n(t), 
    \]
    for $i\in\{1, \cdots, K\}$ and $\hat{R}^n(t) := \min_{1\leq j\leq K} \{\hat{Q}_j^n(0) + \hat{A}_j^n(t) + (\lambda_i^n - \lambda_0 n)t/\sqrt{n} \}$, converges weakly to $\vect{X} = (X_1, \cdots, X_K)^\intercal$ in $D^K[0, T]$ as $n\to\infty$, where $\vect{X}$ satisfies
    \begin{equation}
        \vect{X}(t) = 
        \vect{X}(0) + 
        \vect{\beta} t + 
        \Sigma \vect{W}(t) - 
        R(t)\vect{I},
        \label{limiting process (non-abandonment)}
    \end{equation}
    for all $t\in[0, T]$, where $\vect{\beta}$ and $\Sigma$ are constant vector and diagonal matrix, $\{W_i\}_{1\leq i\leq K}$ are $K$ independent standard Brownian motions, the scalar-valued process
    \begin{equation}
        R(t) = \min_{1\leq j\leq K}\left\{X_j(0)+\beta_j t + \sigma_j W_j(t) \right\},
        \label{R(t) (non-abandonment)}
    \end{equation}
    and $\vect{I} = (1, \cdots, 1)^\intercal\in\mathbb{R}^K$. 
    Moreover, the product $\prod_{j=1}^K X_j(t) = 0$ for any $t\in[0, T]$. 
\end{proposition}


\begin{remark}\label{remark: compare with literature} 
    Our matching system without abandonment (see Proposition \ref{weak convergence theorem (non-abandonment)}) can be thought of as an equivalent asymptotic result as in \cite{plambeck2006optimal} and \cite{gurvich2015dynamic}. In \cite{plambeck2006optimal}'s construction, we may consider a single unlimited order in their shortage process (see their equations (10) and (11)). However, their shortage process would take values at infinity for all time, and it seems non-trivial to obtain an asymptotic result in their analysis. 
    If we consider a single product-production system in \cite{gurvich2015dynamic}, then the optimization problem becomes selecting the best time to formulate matches so as to minimize costs, whose solution would be an instantaneous match. 
    Considering the model proposed in Proposition \ref{weak convergence theorem (non-abandonment)} and following \cite{gurvich2015dynamic}'s construction, we may define the matching matrix $M = (1, \cdots, 1)^\intercal\in \mathbb{R}^K$. In particular, consider the $K=3$ matching system as an instance, where we have matrix $Y$ and the imbalance process $S(\cdot)$
    \[
    Y = \begin{pmatrix}
            1 & 0\\
            -1 & 1\\
            0 & -1
        \end{pmatrix}, 
    \quad 
    S(\cdot) = Y^\intercal (\vect{q_0} + \vect{A}(\cdot)), 
    \]
    where $\vect{q_0}$ denotes initial components in each queue, and $\vect{A}(t)$ represents the arrivals by time $t$, as defined in their equations (3.3) and (1.3), respectively. Meanwhile, 
    by their Theorem 1 and equation (4.5), the queue lengths can be represented as 
    \[
    Y^\intercal \vect{Q}(t) = \begin{pmatrix}
                        Q_1(t) - Q_2(t)\\
                        Q_2(t) - Q_3(t)
                    \end{pmatrix}, 
    \]
    such that $Q_{\text{min}}(t) = 0$ and $Y^\intercal \vect{Q}(t) = S(t)$ for $t\in[0, T]$, where $\vect{Q}\in D^K[0, T]$. In their proof of Theorem 1, we may further construct the matching process $R(t)$ for $t\in[0, T]$ such that $M R(t) = x(t)$, where $x(t) = \vect{q_0} + \vect{A}(t) - \vect{Q}(t) \geq 0$ and $Y^\intercal x(t) = 0$. Observe that our construction with scalar-valued matching completion can be formulated into theirs through the coupling phenomenon; however, the converse needs non-trivial work due to the lack of an explicit expression of $R(\cdot)$ using input processes. 
\end{remark}

The remainder of this section will be divided into four main parts: first, we exhibit the stochastic boundedness, which yields the non-explosive diffusion-scaled queue lengths; 
second, we establish the continuity of integral representation, whose coupling-type integral representation yields a unique yet applicable multivariate coupling relation; 
third, we exhibit the proof of the main result, Theorem \ref{weak convergence theorem}, for completion. Last but not least, 
we discuss a semimartingale decomposition of a special coupled equation obtained from Proposition \ref{weak convergence theorem (non-abandonment)}.

\subsection{Stochastic Boundedness}

Consider the diffusion-scaled processes $(\vect{\hat{Q}^n}, \vect{\hat{A}^n}, \vect{\hat{G}^n}, \hat{R}^n)$ satisfying \eqref{diffusion scaled queue length}. 
Since we know that $\hat{A}_i^n(t) - \hat{R}^n(t) + \frac{\lambda_i^n - \lambda_0 n}{\sqrt{n}}t$ is non-negative for $t\in[0, T]$ and for each $i\in\{1, \cdots, K\}$, $\hat{G}_i^n\geq0$ acts as a negative force that prevents the queue length from walking away and far from the origin. 
Further, the queue lengths are naturally non-negative because of the instantaneous matching behavior, and the stochastic boundedness ensures the non-explosive of queue lengths. 
The following results are basically proved along these facts. 

First, we introduce a new process $\hat{M}_i^n(\cdot)$ for $i\in\{1, \cdots, K\}$ by 
\begin{equation}
    \hat{M}_i^n(t) =\frac{1}{\sqrt{n}}\left(N_i\left(\delta_i^n\int_0^t Q_i^n(s)ds\right) - \delta_i^n\int_0^t Q_i^n(s)ds\right), 
    \label{hat M martingale}
\end{equation}
for all $t\geq0$, where $N_i$ is an unit rate Poisson process introduced in \eqref{abandonment process}. 
By \eqref{diffusion scaled queue length} and \eqref{hat M martingale}, we can rewrite the diffusion-scaled queue length process as
\begin{equation}
    \hat{Q}_i^n(t) = \hat{Q}_i^n(0) + \hat{A}_i^n(t) + \frac{\lambda_i^n - \lambda_0 n}{\sqrt{n}}t - \hat{M}_i^n(t) - \delta_i^n \int_0^t \hat{Q}_i^n(s)ds - \hat{R}^n(t), 
    \label{rewritten diffusion scaled queue length}
\end{equation}
for each $i\in\{1, \cdots, K\}$ and $t\in[0, T]$, where 
\begin{equation}
    \hat{R}^n(t) = \min_{1\leq k \leq N}\left\{\hat{Q}_k^n(0) + \hat{A}_k^n(t) + \frac{\lambda_k^n - \lambda_0 n}{\sqrt{n}}t - \hat{M}_k^n(t) - \delta_k^n \int_0^t \hat{Q}_k^n(s)ds\right\}. 
    \label{rewritten hat R}
\end{equation}
We will show that $\hat{M}_i^n$ is a martingale adapted to $(\mathcal{F}_t^n)_{t\geq0}$ filtration \eqref{natural filtration} in the following Lemma \ref{prove hat M is a martingale}. Therefore, \eqref{rewritten diffusion scaled queue length} yields a martingale representation of the diffusion-scaled queue length process. 

\begin{lemma}
    For any $i\in\{1, \cdots, K\}$, let the assumptions in Theorem \ref{weak convergence theorem} hold. Define $I_i^n := \{I_i^n(t):t\geq0\}$ by
    \begin{equation}
        I_i^n(t) = \delta_i^n \int_0^t Q_i^n(s)ds, 
        \label{random time change process}
    \end{equation}
    where $Q_i^n(\cdot)$ is defined in \eqref{queue length}. Define another filtration condition on the entire arrival processes $\bar{\mathcal{F}}^n = \{\bar{\mathcal{F}}_t^n: t\geq0\}$ by
    \begin{equation}
        \bar{\mathcal{F}}_t^n := \sigma\left(Q_i^n(0), \{A_i^n(u)\}_{u\geq0}, N_i(s): 0\leq s\leq t \text{ and } 1\leq i \leq K\right), \quad t\geq0. 
        \label{filtration for hat M}
    \end{equation}
    Then $(N_i\circ I_i^n)(t) - I_i^n(t)$ is a square-integrable martingale with respect to the filtration  $\bar{\mathcal{F}}_{I_i^n}^n$. Moreover, $\hat{M}_i^n$ is a square-integrable martingale with respect to the filtration $\mathcal{F}_t^n$ as defined in \eqref{natural filtration}, having quadratic variation processes
    \begin{equation}
        \langle\hat{M}_i^n\rangle(t) = \frac{\delta_i^n}{n}\int_0^t Q_i^n(s)ds \quad \text{and} \quad [\hat{M}_i^n, \hat{M}_i^n](t) = \frac{N_i\left(\delta_i^n\int_0^t Q_i^n(s)ds\right)}{n}. 
        \label{quadratic variation process}
    \end{equation}
    \label{prove hat M is a martingale}
\end{lemma}

We prove this by verifying the conditions of Lemma 3.2 in \cite{pang2007martingale}. 
Hence, we defer its proof 
until the appendix. 

With the help of the martingale constructed from the scaled abandonment process, we obtain the moment bound result for the process $\hat{M}_i^n$ below, whose proof is deferred to the appendix, where we mainly apply the Burkholder's inequality (see Theorem 45 in Protter \cite{protter2005stochastic}). 

\begin{proposition}
    Let $T>0$ and for any $i\in\{1, \cdots, K\}$, we have
    \begin{equation}
        E\left[\|\hat{M}_i^n\|_T^2\right]\leq C_1(1+T^l), 
        \label{moment bound for hat M}
    \end{equation}
    where $C_1$ and $l\geq2$ are constants independent of $T$ and $n$. Consequently, the sequence $\{\hat{M}_i^n\}_{n\geq1}$ is stochastically bounded. 
    \label{SB for hat M}
\end{proposition}

Furthermore, the Proposition \ref{SB for hat M} and the martingale representation guarantee the stochastic boundedness of the scaled queue length processes in the following Proposition \ref{moment bound and stochastical boundedness of hat Q}, 
whose proof is delayed to the appendix. 
It is worth mentioning that we employ an interesting technique to manage moment bounds for the scalar-valued non-linear term, utilizing the summation of all entries to dominate the scaled matching completion process. 
This enables us to construct an integral inequality and then employ Gronwall's inequality. 
This technique is also employed in later discussions. 

\begin{proposition}
    Let $T>0$ and for any $i\in\{1, \cdots, K\}$, consider each entry of the state process vector $\vect{\hat{Q}^n}(\cdot)$ in $ D[0, T]$, we have
    \begin{equation}
        E\left[\|\hat{Q}_i^n\|_T^2\right]\leq K(1+K)^2 C_2(1+T^b) \cdot \exp{\left(2c_0(1+K) T \right)}, 
        \label{moment condition on hat Q}
    \end{equation}
    where $C_2$, $b\geq 2$, and $c_0$ are constants independent of $T$ and $n$. Consequently, the sequence $\{\hat{Q}_i^n\}_{n\geq1}$ is stochastically bounded. 
    \label{moment bound and stochastical boundedness of hat Q}
\end{proposition}

\subsection{Continuity of the Integral Representation}

In this section, we intend to establish the continuity of an integral representation, which involves a scalar-valued minimum-type non-linear term. 
This move is analogous to \cite{pang2007martingale}. 
However, the centered and scaled matching completion $\hat{R}^n$ defined in \eqref{diffusion scales} distinguishes the martingale representation from those results in \cite{pang2007martingale}. 
Our major contribution to this section is Theorem \ref{continuity of integral representation} below, where the non-trivial integral representation is a coupled multivariate integral equation with a scalar-valued non-linear term. 


\begin{theorem}
    Consider the integral representation for $t\geq0$,
    \begin{equation}
        \vect{x}(t) = \vect{y}(t) - \int_0^t h(\vect{x}(s))ds - R(t)\vect{I},
        \label{integral representation}
    \end{equation}
    where $\vect{x}$, $\vect{y}\in D^K[0, T]$, $\vect{I} = (1, \cdots, 1)^\intercal\in\mathbb{R}^K$, $h: D^K[0, T]\to D^K[0, T]$ satisfies the Lipschitz condition, and the function $R(\cdot)$ is given by $R(\cdot):= \Psi(\vect{x}, \vect{y})(\cdot)$, where
    \begin{equation}
        \Psi(\vect{x}, \vect{y})(t) = \min_{1\leq j \leq K} \left\{y_j(t) - \int_0^t h_j(\vect{x}(s))ds\right\}. 
        \label{definition of R functional version}
    \end{equation}
    Then, it has a unique solution $\vect{x}$ such that the integral representation constitutes a function $f: D^K[0, T]\to  D^K[0, T]$ mapping $\vect{y}$ into $\vect{x} := f(\vect{y})$. Moreover, $f$ is continuous, provided that the function space $ D^K[0, T]$ is endowed with the topology of uniform convergence over bounded intervals. 
    \label{continuity of integral representation}
\end{theorem}

We prove this result by formulating a contraction mapping. 
Here, our proof is novel in the sense that when establishing the contraction, we tackle the difference between two minimum-type scalar-valued terms by considering indices (depending on time) that attain the minimum in \eqref{definition of R functional version} and further using the summation of all entries as upper bounds to dominate the difference. 
It is also worth mentioning that the proof of Theorem \ref{continuity of integral representation} is mainly in a space endowed with the topology of uniform convergence over bounded intervals since the limiting processes in our discussions have continuous sample paths. 
However, a space endowed with Skorokhod $J_1$ topology works better than uniform topology in general cases since the latter may bring in measurability issues (see Section 11.5.3 in \cite{whitt2002stochastic}). 
A more detailed proof regarding the continuity in the Skorokhod space $ D^K[0, T]$ endowed with the Skorokhod $J_1$ topology can be found in 
Section 3.8 of \cite{xie2022topics}. 

\subsection{Proof of Theorem \ref{weak convergence theorem}}

\begin{proof}[Proof of Theorem \ref{weak convergence theorem}]
The proof is mainly divided into two parts: first, we intend to show the coupled stochastic integral equation \eqref{limiting process (model with abandonment)} admits a unique strong solution; second, we will show that $\vect{\hat{Q}^n}$ is convergent weakly and furthermore, the limiting process satisfies \eqref{limiting process (model with abandonment)}. 

Consider a functional $\Lambda:  C^K[0, T]\to C^K[0, T]$ given by
\begin{equation}
    \Lambda(\vect{Y})(t) = \vect{X}(0) + 
    \vect{\beta} t + 
    \Sigma \vect{W}(t) -
    \int_0^t \vect{\delta} \diamond \vect{X}(s)ds - 
    R_{\vect{Y}}(t)\vect{I},
\end{equation}
where $\vect{Y}(\cdot)=(Y_1(\cdot),\cdots Y_K(\cdot))^\intercal\in C^K[0, T]$ with $\vect{Y}(0) = \vect{X}(0)$, parameters are as defined in \eqref{limiting process (model with abandonment)}, and the process $R_{\vect{Y}}(\cdot)$ is given by
\begin{equation}
    R_{\vect{Y}}(t) = \min_{1\leq k\leq K}\left\{x_k +\beta_kt + \sigma_kW_k(t) - \int_0^t \delta_k Y_k(s)ds\right\},
\end{equation}
for $t\geq0$. To show the existence and uniqueness of a strong solution to \eqref{limiting process (model with abandonment)}, it suffices to show the functional $\Lambda$ admits a unique fixed point. 
Following the main idea of proof mentioned above in Theorem \ref{continuity of integral representation}, we can obtain a Lipschitz property that for any $\vect{Y}^{(1)}, \vect{Y}^{(2)}\in C^K[0, T]$ such that $\vect{Y}^{(1)}(0) = \vect{Y}^{(2)}(0) = \vect{X}(0)$, 
\begin{equation}
    \|\Lambda(\vect{Y}^{(1)})- \Lambda(\vect{Y}^{(2)})\|_T \leq \epsilon(T) \|\vect{Y}^{(1)} - \vect{Y}^{(2)}\|_T, 
\end{equation}
where $\epsilon(T) = (1+K)\sqrt{K}\left(\sup_{1\leq k \leq K} (\delta_k)\right)T$. 
We partition the time interval $[0, T]$ into several subintervals with length $T_1>0$ such that $\epsilon(T_1)<1$ to formulate successive contraction mappings over each subinterval (see also, proof of Theorem \ref{continuity of integral representation}). 
Hence, we can conclude that $\Lambda$ is a contraction mapping by taking suitable $T$ consecutively. 
By the Banach fixed-point theorem, $\Lambda$ admits a unique fixed point, which suggests that there exists a unique strong solution $\vect{X}$ to \eqref{limiting process (model with abandonment)}. 
Additionally, the existence and uniqueness of a weak solution to the general state-dependent coupled stochastic integral equation and its strong Markov property can be found in Section 3.5 of \cite{xie2022topics}.

Next, we intend to show that $\vect{\hat{Q}^n}$ is convergent weakly, and the heavy traffic limiting process satisfies \eqref{limiting process (model with abandonment)}. 
By the martingale representation of the diffusion-scaled queue length \eqref{rewritten diffusion scaled queue length}, we define a new functional $F: D^K[0, T]\to D^K[0, T]$ such that 
\begin{equation}
    \vect{\hat{Q}^n}(t) = F(\vect{\xi}^n)(t), 
\end{equation}
where $\vect{\xi^n}(\cdot) = (\xi_1^n(\cdot), \cdots, \xi_K^n(\cdot))^\intercal$ with entries defined by 
\begin{equation}
    \xi_k^n(t) := \hat{Q}_k^n(0) + \frac{\lambda_k^n - \lambda_0 n}{\sqrt{n}}t + \hat{A}_k^n(t) - \hat{M}_k^n(t)
\end{equation}
for $t\geq0$ and $k\in\{1, \cdots, K\}$.
Since $\vect{\xi^n}\Rightarrow\vect{\xi}$ in $ D^K[0, T]$, where $\vect{\xi}(\cdot)=(\xi_1(\cdot), \cdots, \xi_K(\cdot))^\intercal$ with entries given by $\xi_k(t) = x_k + \beta_k t + \sigma_k W_k(t)$ for $t\geq0$, and using the continuous mapping theorem and Theorem \ref{continuity of integral representation}, we have
\begin{equation}
    F(\vect{\xi^n}) \Rightarrow F(\vect{\xi}) := \vect{Q}
\end{equation}
in $D^K[0, T]$ as $n\to\infty$. By the existence of a unique solution to \eqref{limiting process (model with abandonment)}, we obtain that $\vect{Q}(\cdot)$ coincides with $\vect{X}(\cdot)$. This completes the proof. 
\end{proof}


We can further extend the weak convergence result in Theorem \ref{weak convergence theorem} to the convergence in $L^p$ sense for some appropriate $p$ values on a special probability space with the help of the Skorokhod device. 
This works well due to the fact that we are considering a Markovian model. 
We present such an extension in Corollary \ref{L^p extension} below and postpone its proof to the appendix, which mainly relies on verifying the uniform integrability of a proper integrand and employing Burkholder's inequality. Meanwhile, we carefully exhibit an upper bound of the scalar-valued non-linear term, which is analogous to the technique used in Theorem \ref{continuity of integral representation}. 

\begin{corollary}
    The weak convergence in Theorem \ref{weak convergence theorem} can be refined on a special probability space under which we have for $p\geq 1$, 
    \begin{equation}
        E\left[\|\vect{\hat{Q}^n} - \vect{X}\|_T^p\right]\to 0
    \end{equation}
    as $n\to\infty$. 
    \label{L^p extension}
\end{corollary}

\subsection{Semimartingale Property}
\label{sec: Semimartingale Property of the Coupled Process}

The characteristics of the coupled stochastic integral equations are intriguing. 
Intuitively, when a queue becomes empty after a match, it may remain empty for a certain time due to the fullness of other queues since matching is instantaneous. This queue may stick to zero for a certain time. 
Moreover, when more than one queue is empty, the above situation could be transferred to other queues. Such a stickiness behavior also persists in the heavy-traffic limiting processes. 
We close this section by
revealing the semimartingale property of a special coupled stochastic integral equation \eqref{limiting process (non-abandonment)}, which reveals that there are some underlying local time processes from the scalar-valued process $R$. 
However, the semimartingale property of the more general coupled equation \eqref{limiting process (model with abandonment)} is much more complex, and we will discuss it in further studies. 

\begin{proposition}
    The heavy traffic limit $\vect{X}(t) = (X_1(t), \cdots, X_K(t))^\intercal$ obtained in \eqref{limiting process (non-abandonment)} is a semimartingale for $t\geq 0$.
    \label{Theorem semimartingale property}
\end{proposition}

We postpone its proof to the appendix. Its derivation of the semimartingale decomposition is novel and applicable since the technique of successively employing Tanaka's formula occurs when tackling a maximum of many entries.
For brevity, we exhibit the decomposition for $X_1$ in the case of $K\geq 2$ categories as an example, where we have 
\begin{equation}
\begin{aligned}
    X_1(t) &= X_1(0) - X_K(0) + \sum_{l=1}^{K-1} (Y_l(0))^+ - \sqrt{\sigma_1^2+\sigma_K^2}B_{1K}(t) + (\beta_1-\beta_K)t\\
    &\quad +\sum_{l=1}^{K-1}\sqrt{\sigma_l^2 + \sigma_K^2}\int_0^t \mathbbm{1}_{[Y_l(s)>0]} \prod_{j=1}^{l-1}\left(1-\mathbbm{1}_{[Y_j(s)>0]}\right) dB_{lK}(s)\\
    &\quad +\sum_{l=1}^{K-1} (\beta_K-\beta_l)\int_0^t \mathbbm{1}_{[Y_l(s)>0]} \prod_{j=1}^{l-1} \left(1-\mathbbm{1}_{[Y_j(s)>0]}\right)ds \\
    &\quad + \frac{1}{2}\sum_{l=1}^{K-1} \int_0^t \prod_{j=1}^{(K-1)-l} \left(1-\mathbbm{1}_{[Y_j(s)>0]}\right)dL_s^{(l)}, 
    \label{semimartingale property (non-abandonment)}
\end{aligned}
\end{equation}
for $t\geq 0$, where $\{Y_l(t)\}_{1\leq l\leq K-1}$ are defined as the following iterations: 
\begin{align*}
    Y_{K-1}(t) &= -\xi_{K-1}(t)+\xi_K(t), & \eta_{K-1}(t) &= -\xi_{K}(t) + (Y_{K-1}(t))^+, \\
    Y_{K-2}(t) &=-\xi_{K-2}(t) - \eta_{K-1}(t), & \eta_{K-2}(t) &= \eta_{K-1}(t) + (Y_{K-2}(t))^+\\
    &\quad \vdots & &\quad \vdots\\ 
    Y_1(t) &= -\xi_1(t) - \eta_2(t), & \eta_1(t) &= \eta_2(t) + (Y_1(t))^+\\
    X_1(t) &= \xi_1(t) + \eta_1(t), &  \xi_j(t) &= X_j(0) + \beta_j t + \sigma_j W_j(t)
\end{align*}
for $j \in \{1, \cdots, K\}$. 
Moreover, $L_t^{(l)}$ is the local time process for $Y_{K-l}(t)$ at the origin for $l \in \{1, \cdots, K-1\}$ and $\{B_{lK}(\cdot)\}_{1\leq l\leq K-1}$ are $K-1$ mutually correlated Brownian motions. For any $l$, $B_{lK}$ is a Brownian motion depending on two independent standard Brownian motions $W_l$ and $W_K$. 

Although the semimartingale decomposition is quite complex, one can still observe that the non-linear term $R(\cdot)$ involves some underlying local time processes. 
A more concrete example appears in double-ended matching queue systems for the case of $K=2$, see for example, \cite{liu2021admission}.

\section{Asymptotic Little's Law}\label{sec: little's law}

In this section, we establish an asymptotic relationship between a queue length process and its corresponding virtual waiting time process, which is called Little's law in the literature. 
The purpose of this asymptotic relationship is straightforward since it enables the system manager to estimate the queue lengths provided the information about customers' waiting times without having access to observe the queue lengths. 
Such a circumstance is normal in telecommunication centers, and its application to matching queue systems is moot (cf. \cite{gans2003telephone}, \cite{mandelbaum2012queues}). 
Since the components in this model are perishable, the virtual waiting time is quite complicated. 
Due to those abandoned components, the order is no longer preserved; namely, a component $j$ of category $i$ may not match with the $j$th components from other categories. 
We will exhibit an explicit expression for the virtual waiting time process in \eqref{virtual waiting time} below. 
Here we intend to show that for each $i\in\{1, \cdots, K\}$, 
\begin{equation}
    \|\hat{Q}_i^n - \lambda_0 \hat{V}_i^n\|_T \to 0 
\end{equation}
in probability as $n\to\infty$, where $\hat{V}_i^n(\cdot):= \sqrt{n}V_i^n(\cdot)$ is the diffusion-scaled virtual waiting time. 

We introduce the virtual waiting time process $V_i^n(t)$ for $i\in\{1, \cdots, K\}$ as the amount of time an infinite patience hypothetical component of category $i$ would have to wait had it arrived at time $t\in[0, T]$, which is given by
\begin{equation}
    V_i^n(t) := \sum_{k=1}^{A_i^n(t)} v_{ik}^n - \int_0^t \mathbbm{1}_{[V_i^n(s)>0]}(s)ds, 
    \label{virtual waiting time}
\end{equation}
where $v_{ik}^n$ represents the amount of time the $k$th component spent in the head position of queue $i$. Notice that if the component $k$ of the $i$th category is abandoned before reaching the head position, we impose $v_{ik}^n := 0$. 
However, if it reaches the first place of a queue, then $v_{ik}^n > 0$, and it may either abandon the system or get matched from there after $v_{ik}^n$ units of time. 
It also defines the amount of workload needed to empty the queue provided no new arrivals after time $t$. 
The definition \eqref{virtual waiting time} is not used in our proof, but we intend to present its precise definition so that we can get a clear picture of the behaviors of the $i$th queue as well as its profiles in later discussions. 
A similar definition has been used in \cite{liu2021admission} for a double-ended queueing model. 

Next, we present the main result (see Theorem \ref{little's law}) of this section, whose proof follows a similar idea as in Theorem 4.5 of \cite{weerasinghe2014diffusion}. 
Hence, we defer its proof to the appendix. 

\begin{theorem}
    Under the assumptions of Theorem \ref{weak convergence theorem} and let $T>0$, we have
\begin{equation}
    \sup_{1\leq i\leq K}\|\hat{Q}_i^n - \lambda_0 \hat{V}_i^n\|_T \to 0 
\end{equation}
in probability as $n\to\infty$, for each $i\in\{1, \cdots, K\}$. 
\label{little's law}
\end{theorem}

We close this section by exhibiting an interesting moment bound result (see Proposition \ref{SB of hat V (non-abandonment)} below) regarding the virtual waiting time of the matching queue with no abandonment as proposed in Proposition \ref{weak convergence theorem (non-abandonment)}.

\begin{remark}
Consider the matching queue model with no abandonment as proposed in Proposition \ref{weak convergence theorem (non-abandonment)}. 
We observe that the order is preserved; namely, a component $j$ will certainly match with the $j$th component from other categories since none of those components can leave the system without a match. 
However, in terms of the model proposed in Theorem \ref{weak convergence theorem}, the system loses such a benefit due to the perishable components. 
Alternatively, one may consider employing the expressions of the virtual waiting times to rule out the components that may eventually abandon its queue, but a proper stochastic upper bound needs to be determined, which is a challenging problem due to the intractable matching operation. 
We will study this in future work. 
\end{remark}


\begin{proposition}
    Let the assumptions in Proposition \ref{weak convergence theorem (non-abandonment)} hold and for each $i\in\{1, \cdots, K\}$, we have that $\hat{V}_i^n$ is stochastically bounded and as a consequence, $\|V_i^n\|_T \to 0$ in probability as $n\to\infty$. In addition, we have 
    \begin{equation}
        \sup_{n\geq1}E\left[\abs{\hat{V}_i^n(t)}^2\right] \leq C_V(1+T^b), 
        \label{moment bound for hat V and each t (non-abandonment)}
    \end{equation}
    for each $t\in[0, T]$, where $C_V$ and $b\geq 2$ are constants independent of $T$ and $n$. 
    \label{SB of hat V (non-abandonment)}
\end{proposition}

Its novel proof can be found in the appendix, where we mainly employ the benefits of the order-preserving property. 
It is straightforward to deduce the stochastic boundedness if we can obtain the moment bounds for the supremum normal (for instance, see Proposition \ref{moment bound and stochastical boundedness of hat Q}). However, this result does not hold for the supremum norm over $[0, T]$ and instead, we have a bound for each time $t$ in \eqref{moment bound for hat V and each t (non-abandonment)} utilizing the interesting property. 
For matching models and their virtual waiting times, such an approach of obtaining moment bounds can be applied as long as the order is preserved. 
\section{Convergence of Cost Functionals}\label{sec: convergence of cost functionals}

It is quite natural that such matching queues are equipped with cost structures. 
In this section, we intend to address the convergence of performance measures associated with various cost functions under a cost structure containing a holding cost for storing components in queues and a penalty cost for abandoned components. 
We find that there is an unusual restriction of the discount factor that depends on the number of categories $K$ if we consider the infinite-horizon discounted cost functional. 
These will provide possible applications of the asymptotic analysis performed in Section \ref{sec: weak convergence} for cost minimization problems.

\subsection{Cost Structure}

We introduce an infinite-horizon discounted cost functional $J(\vect{\hat{Q}^n}(0), \vect{\hat{Q}^n})$ associated with the $n$th matching queue system introduced in \eqref{diffusion scaled queue length}, which consists of two types of costs: a holding cost generated by storing components in queues and a penalty cost proportional to the number of abandoned components from each category. 
Let $C\in C(\mathbb{R}_+^K)$ be a non-negative holding cost function that characterizes the holding cost per time unit, and let $p_j>0$ represent the cost incurred per abandoned components from category $j\in\{1, \cdots, K\}$ per time unit. Here, $p_j$ can be interpreted as the cost rate of abandonment per component of category $j$ per time unit. 
Let $\gamma > 0$ represent the interest rate. 
The infinite-horizon discounted cost functional is given by
\begin{equation}
    J(\vect{\hat{Q}^n}(0), \vect{\hat{Q}^n}) = E\left[ \int_0^{\infty} e^{-\gamma s}\left(C(\vect{\hat{Q}^n}(s))ds + \sum_{j=1}^K p_j d\hat{G}_j^n(s)\right)\right], 
    \label{cost functional for nth system}
\end{equation}
where $\gamma$ and $p_j$ are positive constants. 
Here, we restrict the parameter $\gamma > 2l c_0(1+K)$ for $l\geq1$, where $c_0>0$ is a constant satisfying $\sup_{1\leq k\leq K} (\delta_k^n) \leq c_0$ as in \eqref{moment condition on hat Q} so that it ensures the uniform integrability in later discussions. 
Here, the value $l$ is associated with appropriate growth conditions of the cost function. 
Moreover, one can observe that this restriction can be fulfilled only for large interest rate $\gamma$ when $K$ is large. 

Our goal is to establish the convergence of cost functional under some growth conditions for the cost function $C(\cdot)$ such that
\begin{equation}
    \lim_{n\to\infty} J(\vect{\hat{Q}^n}(0), \vect{\hat{Q}^n}) = J(\vect{x}, \vect{X}), 
\end{equation}
where $\vect{X}$ is the limiting diffusion process obtained in Theorem \ref{weak convergence theorem}, and 
\begin{equation}
    J(\vect{x}, \vect{X}) = E\left[\int_0^{\infty} e^{-\gamma s}\left( C(\vect{X}(s))ds + \sum_{j=1}^K p_j\delta_j X_j(s)ds\right)\right]. 
    \label{cost functional for limiting process}
\end{equation}
In the rest of this section, we consider a concrete example of cost functions: a holding cost with polynomial growth.

\subsection{Convergence Results}

Consider a cost function of the form $C(\cdot) = (C_1(\cdot), C_2(\cdot), \cdots, C_K(\cdot))$, where $C_j:\mathbb{R}_+ \to [0, \infty)$ for $1\leq j\leq K$ are continuous with polynomial growth, namely
\begin{equation}
    0\leq C_j(x)\leq c_j(1+\abs{x}^p),  
    \label{cost function polynomial growth}
\end{equation}
where $c_j>0$ is a constant and $1\leq p < 2l$ for $l\geq1$. 
Under the same cost structure introduced above, the infinite-horizon discounted cost functional \eqref{cost functional for nth system} associated with the $n$th matching queue system in Theorem \ref{weak convergence theorem} can be written as
\begin{equation}
    J(\vect{\hat{Q}^n}(0), \vect{\hat{Q}^n}) = E\left[\sum_{j=1}^K \int_0^{\infty} e^{-\gamma s} \left(C_j(\hat{Q}_j^n(s))ds + p_j d\hat{G}_j^n(s)\right)\right], 
    \label{cost functional for nth system (polynomial cost)}
\end{equation}
We also introduce an infinite-horizon discounted cost functional $J(\vect{x}, \vect{X})$ associated with the limiting process obtained in Theorem \ref{weak convergence theorem} by
\begin{equation}
    J(\vect{x}, \vect{X}) = E\left[\sum_{j=1}^K \int_0^{\infty} e^{-\gamma s} \left(C_j(X_j(s)) + p_j\delta_j X_j(s)\right)ds\right]. 
    \label{cost functional for limiting process (polynomial cost)}
\end{equation}

Our objective is to show the expected value of the infinite-horizon discounted cost functional of the coupled queueing systems converges to that of the limiting processes.

\begin{theorem}
    Consider the sequence of matching queue processes $(\vect{\hat{Q}^n})$ converges weakly to the diffusion process $\vect{X}$ as described in Theorem \ref{weak convergence theorem}, we have
    \begin{equation}
        \lim_{n\to\infty} J(\vect{\hat{Q}^n}(0), \vect{\hat{Q}^n}) = J(\vect{x}, \vect{X}),  
        \label{convergence of cost functional}
    \end{equation}
    where $J(\vect{\hat{Q}^n}(0), \vect{\hat{Q}^n})$ and $J(\vect{x}, \vect{X})$ are the cost functionals defined in \eqref{cost functional for nth system (polynomial cost)} and \eqref{cost functional for limiting process (polynomial cost)}, respectively. 
    \label{Theorem convergence of cost functionals (polynomial)}
\end{theorem}

We delay its proof to the appendix, where we mainly verify the uniform integrability of appropriate integrands by considering the expectations separately under some restrictions.

Some comments on Theorem \ref{Theorem convergence of cost functionals (polynomial)} and its proof are in order. First, we recall the growth condition for the holding cost function $C(\cdot)$, which is given by $0\leq C_j(x)\leq c_j(1+\abs{x}^p)$ for $1\leq p< 2l$ with $l\geq1$ assumed in \eqref{cost function polynomial growth}. 
Since we can pick any $l\geq 1$, the polynomial growth assumption can be extended to any order. 
If we pick $l=1$, our result could be weakened to the case of polynomial growth with $1\leq p<2$, which holds in general due to the second-moment condition of the diffusion scaled centered general arrival processes as described in \eqref{moment condition on hat A}. 
However, for the Markovian matching queue model, we may derive a higher-ordered moment condition for the diffusion-scaled-centered Poisson arrivals, 
which contributes to the growth condition of the cost function with higher orders. 
Second, the assumption for the interest rate $\gamma$, which is given by $\gamma > 2lc_0(1+K)$ with $l\geq1$, guarantees uniform integrability.
This restriction only allows us to take large $\gamma$ when the amount of categories $K$ is large. 
Ideally, we desire to relax such a restriction by formulating a non-trivial higher-order 
(or even exponential)
moment 
estimate
that is similar to the one in Proposition \ref{moment bound and stochastical boundedness of hat Q}, but 
with a tighter upper bound.
Third, it is natural to consider an extension of more general holding costs. For instance, we may assume an arbitrary non-decreasing holding cost and impose a convexity condition. However, the uniform integrability of proper integrands needs to be fulfilled. 
Due to the depreciation in practice, we consider the infinite-horizon discounted cost, but other performance measures can be established per interest. 
We will consider more concrete control problems in our follow-up paper \cite{xiewu2023control}. 
Additionally, a performance measure computation of some special coupled processes and distribution functions can be found in Section 3.3.3 of \cite{xie2022topics}. 

\section{Numerical Simulations}\label{sec: numerical simulations}
In this section, we intend to perform numerical simulations of some interesting cases of the above-proposed matching queue system, and our goal is to understand the dynamics of the queue length processes in heavy traffic. 
To this end, we discretize the limiting process \eqref{limiting process (model with abandonment)} and employ the trapezoidal rule to tackle the integral term so that for a time partition $0= t_0 < t_1 < \cdots < t_N = T$ with time length $\Delta t = T/N$, $N=250$ and $T=1$, at time $(t_j)_{j=1}^N$, 
\begin{equation}
    \vect{X}(t_j) = \vect{X}(0) + \beta t_j + \Sigma \vect{W}(t_j) - \vect{G}(t_j) - R(t_j)\vect{I}, 
\end{equation}
where $\vect{G}(t_j) = (G_1(t_j), \cdots, G_K(t_j))^\intercal$ and for $i\in\{1, \cdots, K\}$, 
\begin{equation}
    G_i(t_j) = \delta_i \sum_{k=1}^j \frac{\Delta t}{2} (X_j(t_{k-1}) + X_j(t_k)). 
\end{equation}
Moreover, the matching completion process $R(t_j)$ is as defined in \eqref{limiting process (model with abandonment)}. 
In our simulations, we intend to let the diffusion coefficients be identical and study the behaviors of queues along with the number of categories and drifts. 

\begin{example}[Variation of category number $K$]
\label{example 1}
    Intuitively, if we have many categories $K$, the number of matchings should be relatively less than that with small categories. To demonstrate such a relationship, we take different numbers of categories $K \in \{2, 3, 4, 5, 20, 80, 500\}$. We assume initial states $x_i\in[0, 0.02]$, diffusion coefficients $\sigma_i$'s are fixed constant 2, drifts $\beta_i$'s and abandonment rates $\delta_i$'s are randomly generated over interval $[-0.3, 0.1]$ and $[0.01, 1]$, respectively. 
    We observe in Figure \ref{fig: matching completion} that since the matching completion $\hat{R}^n(\cdot)$ process is centered as defined in \eqref{diffusion scales}, more categories yield fewer matchings, which further renders the limiting processes $R(\cdot)$ rapidly decreasing. 
    \begin{figure}[h!]
        \centering
        \includegraphics[width=0.6\linewidth]{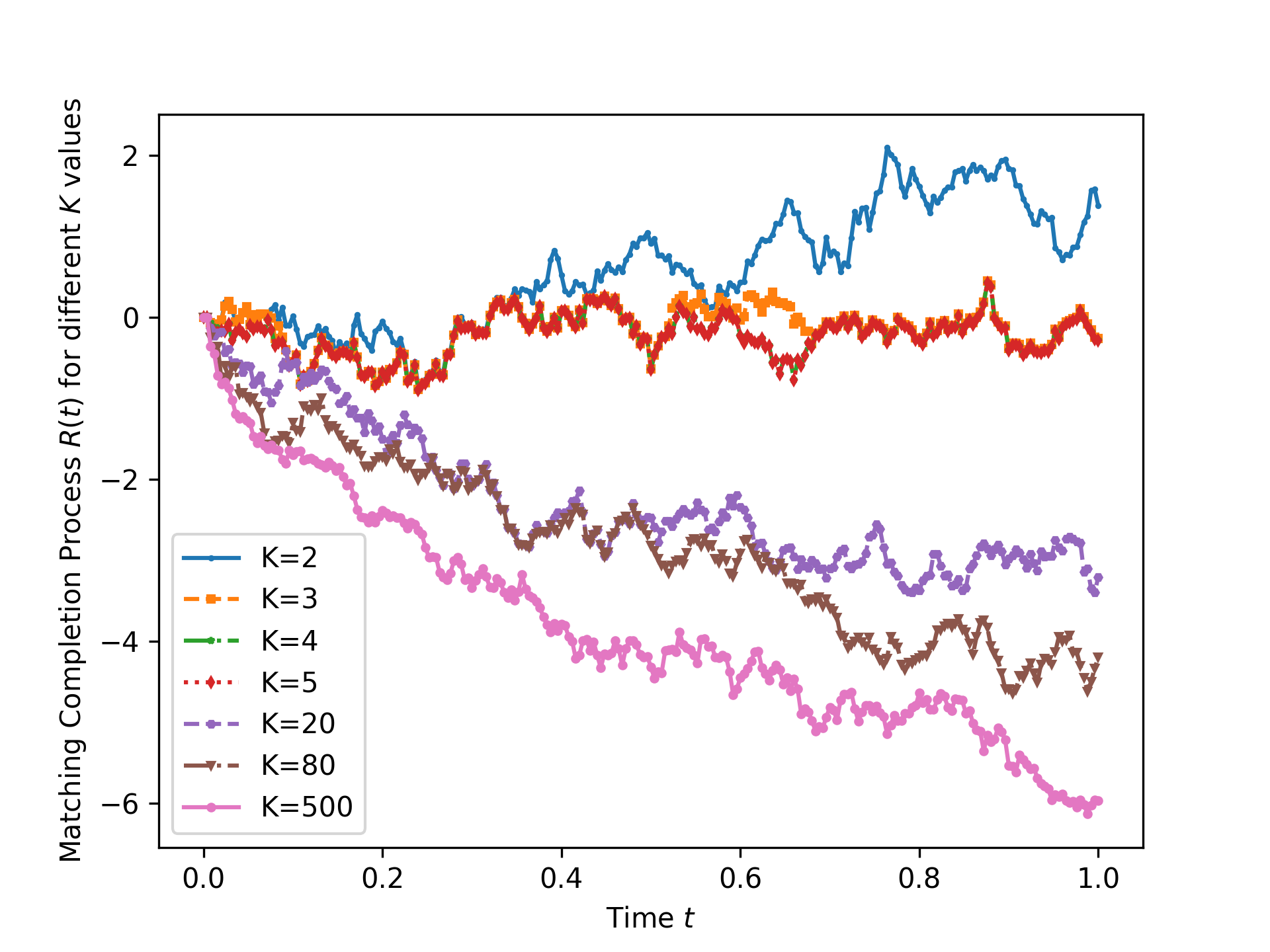}
        \caption{Matching completion process $R$ for different $K$ values}
        \label{fig: matching completion}
    \end{figure}
    Fewer matchings also mean components have to stay in their queues for a longer time, which may generate more abandonment. For instance, we consider the first category in Figure \ref{fig: abandonment}, and the other classes follow a similar fashion. 
    \begin{figure}[h!]
        \centering
        \includegraphics[width=0.6\linewidth]{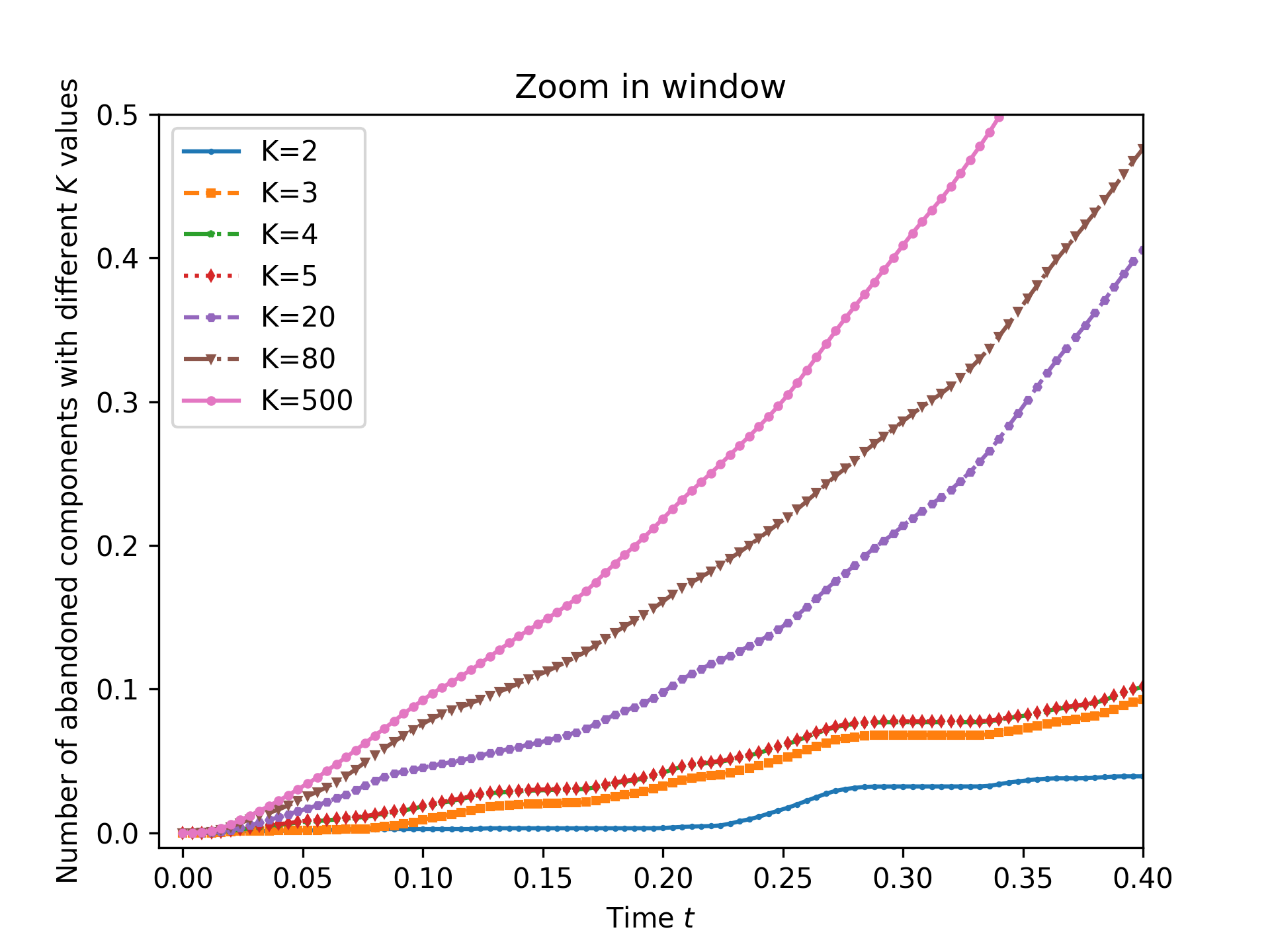}
        \caption{Number of abandoned components in the first category queue for different $K$ values}
        \label{fig: abandonment}
    \end{figure}
\end{example}

\begin{example}[Stickiness as category number $K$ increases]
    We are interested in observing the dynamics of stickiness as $K\to\infty$. Here, by stickiness, we mean the behavior of queue lengths sticking at the origin for a random time, which depends on the profiles of other queues. In Figure \ref{fig: queue lengths for different categories}, we conduct simulations for different numbers of categories, namely, $K = 2, 3, 4, 5$. Here, we assume the same system parameters as in Example \ref{example 1}.  
    One can observe that when $K=3$, the stickiness of queue $X_2$ becomes less often due to the presence of queue $X_3$, and this observation remains true for the rest of cases $K=4, 5$. 
    Since drifts and abandonment rates for the same queue are identical throughout the variation of $K$'s, one may observe similar shapes of paths for identical queues, for instant queue 1, $X_1$. 
    \begin{figure}[h!]
        \centering
        \subfigure{\includegraphics[width=0.46\textwidth]{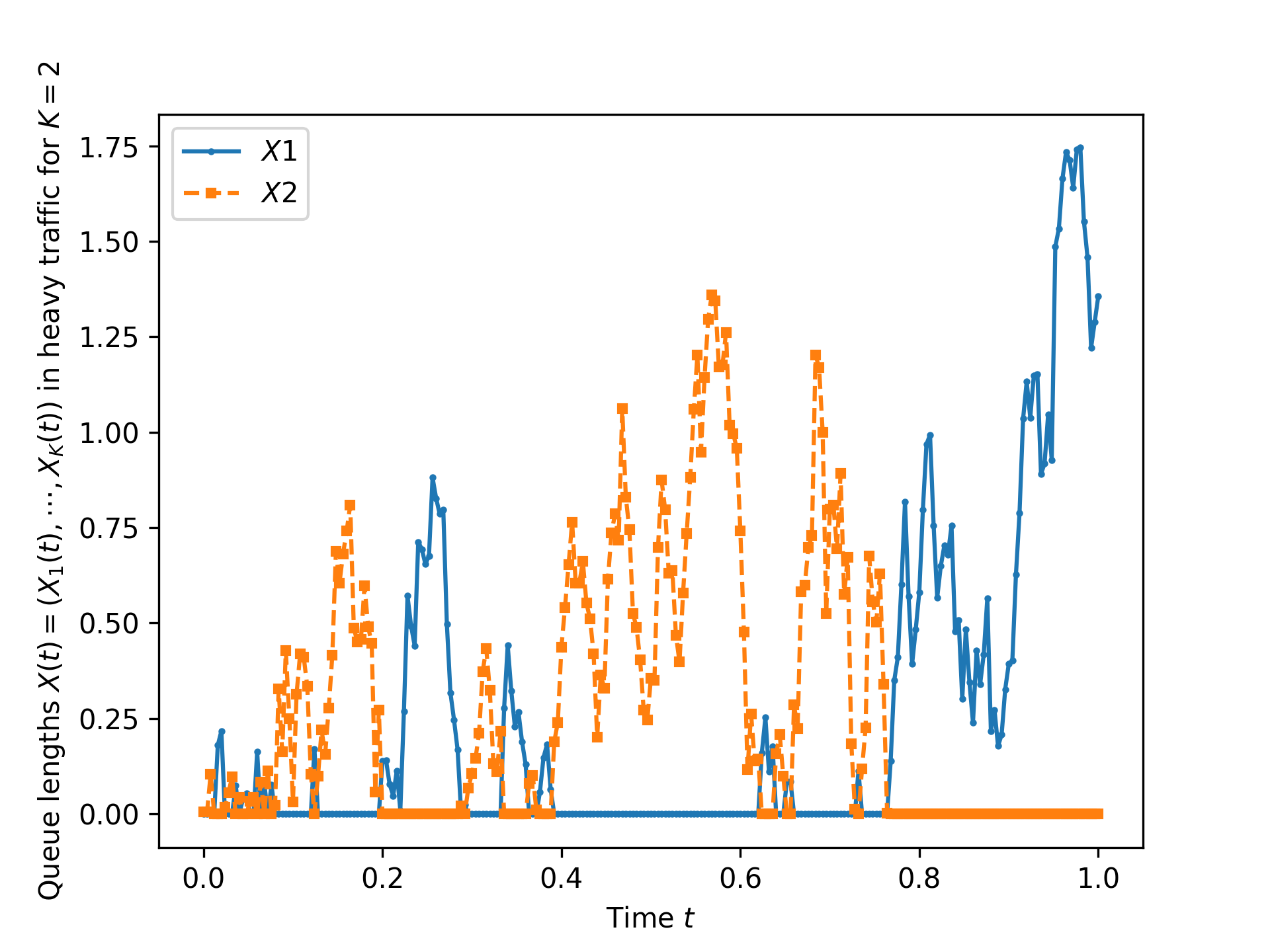}} 
        \subfigure{\includegraphics[width=0.46\textwidth]{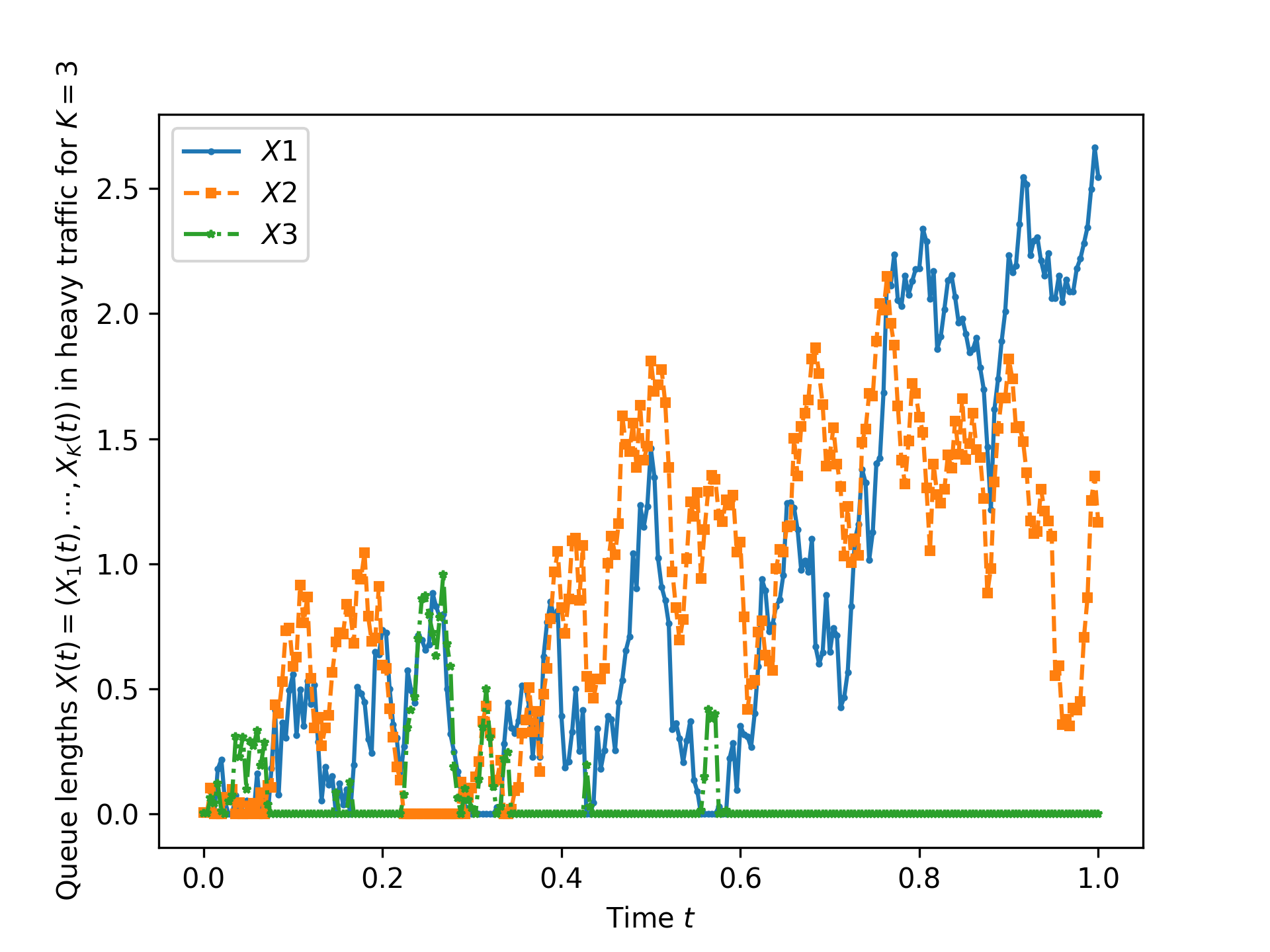}} 
        \subfigure{\includegraphics[width=0.46\textwidth]{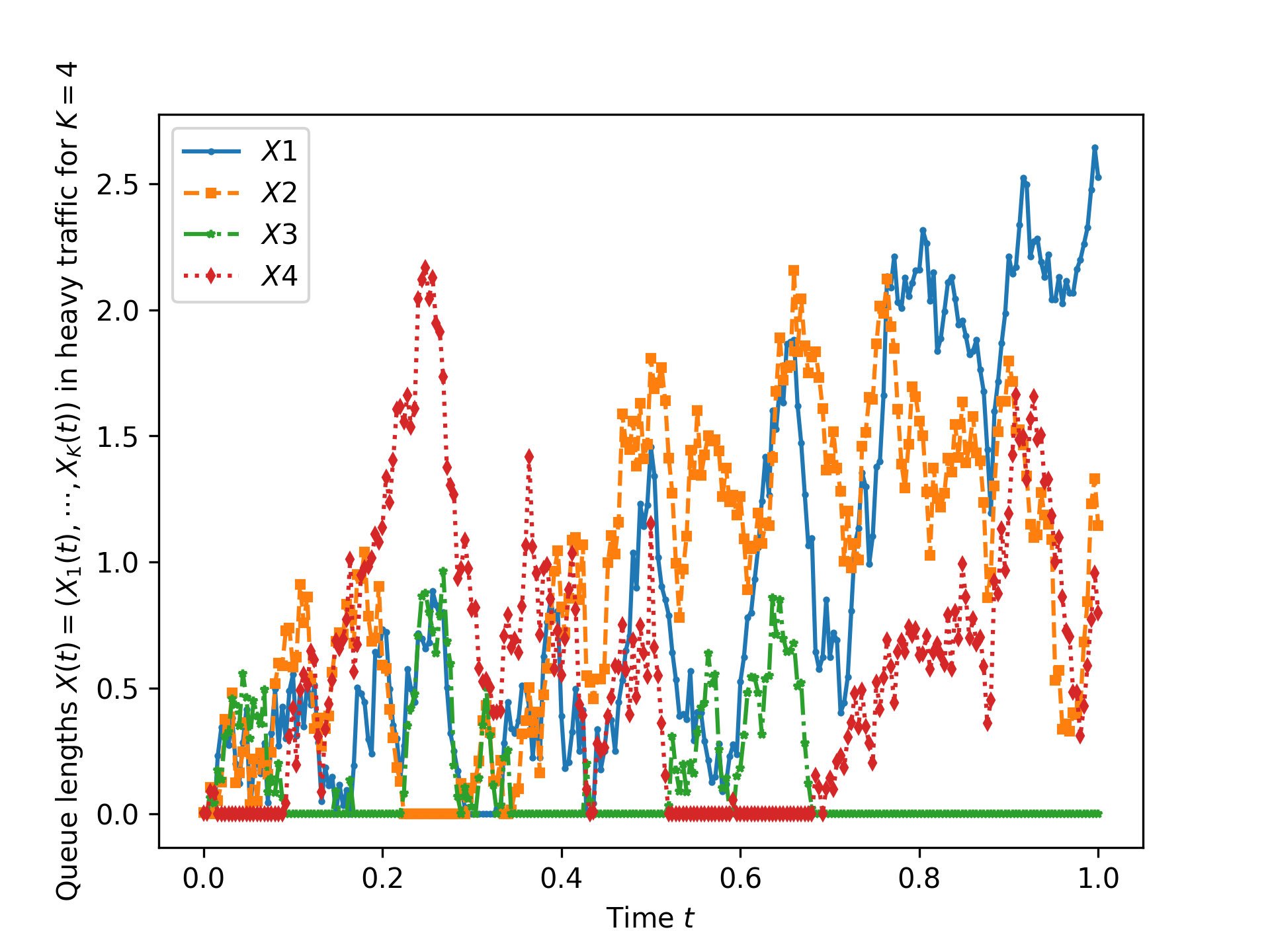}}
        \subfigure{\includegraphics[width=0.46\textwidth]{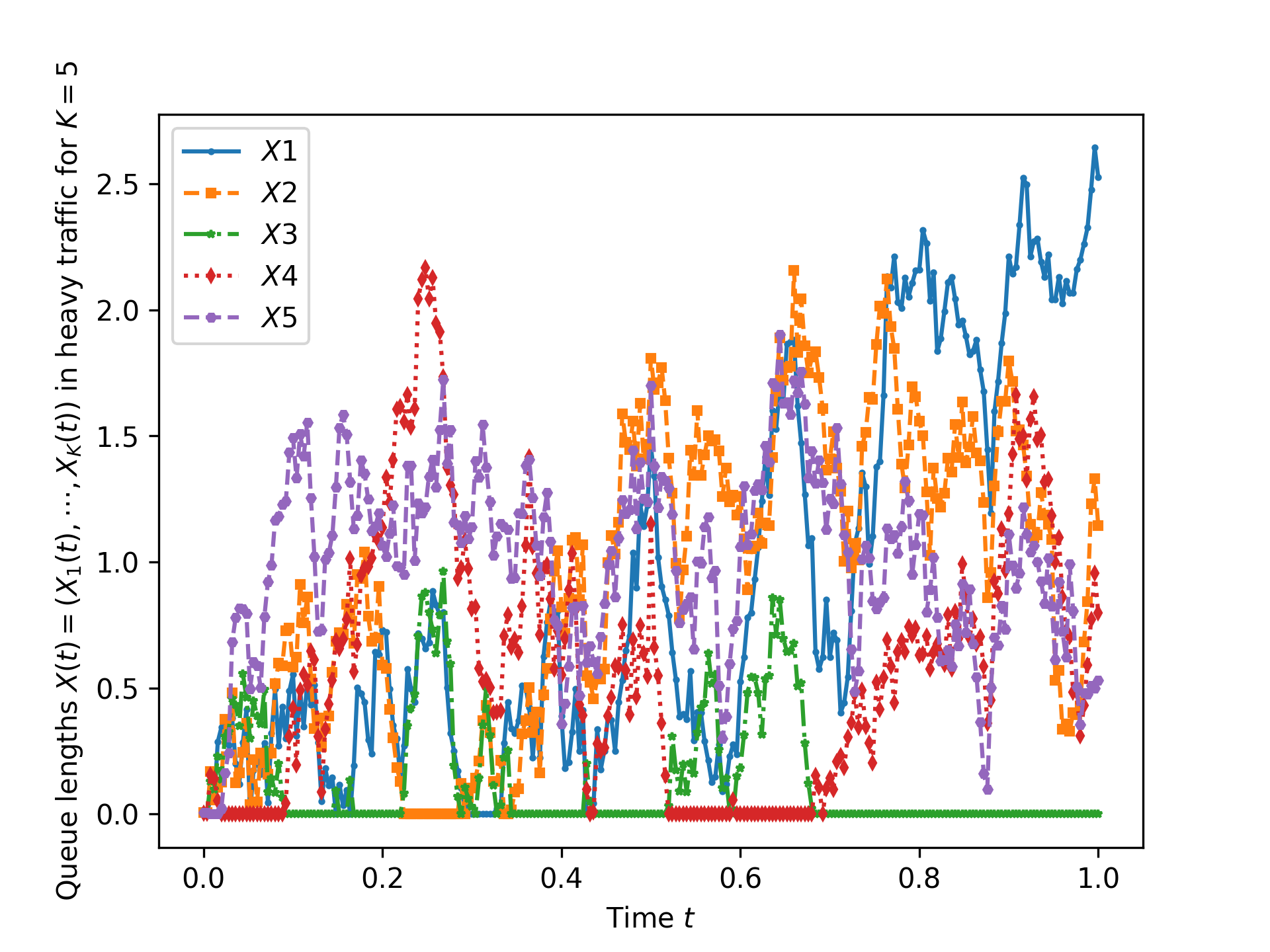}}
        \caption{Queue lengths in heavy traffic for different $K$ values}
        \label{fig: queue lengths for different categories}
    \end{figure}
    As we conduct the simulation for more than five categories so as to let $K\to\infty$, we can observe less stickiness, which indicates the fraction of time a queue remains empty is decaying since other new component queues share the status of empty queues. However, its declining velocity remains unknown, and we will analyze more about this in further studies. 
    
    Intuitively, the drift coefficients should dominate this quantity since they can provide positive or negative forces over time. 
    In Figures \ref{fig: queue lengths for different drifts}, we consider a heavy traffic limit of queue lengths with $K=4$ categories and different drift coefficient $\beta_3$ for the third queue as an example. 
    One can observe that depending on the value for $\beta_{3}$, the proportion of times the corresponding queue stays at zero less often for higher drifts. Such a relationship can be found in the following Table \ref{tab: queue lengths for different drifts} for different drift $\beta_i\in\{-4, -3, -2, 1, 2, 4, 6\}$ for $i\in\{1, 2, 3, 4\}$. 
    \begin{table}[h]
    \begin{center}
        \caption{Proportion of time queue $X_i$'s stays at zero for different drift $\beta_i$ values}
        \label{tab: queue lengths for different drifts}
        \begin{tabular}{@{}cccccccc@{}}
        \toprule
        $\beta_1$ & -4    & -3   & -2    & 1     & 2     & 4     & 5  \\ 
        \midrule
        $X_1$ & 0.6 & 0.552 & 0.448 & 0.036 & 0.036 & 0.024 & 0.02 \\ 
        \botrule
        $\beta_2$ & -4    & -3   & -2    & 1     & 2     & 4     & 5  \\ 
        \midrule
        $X_2$ & 0.756 & 0.656 & 0.468 & 0.008 & 0.004 & 0.0 & 0.0 \\ 
        \botrule
        $\beta_3$ & -4    & -3   & -2    & 1     & 2     & 4     & 5  \\ 
        \midrule
        $X_3$ & 0.984 & 0.98 & 0.976 & 0.748 & 0.652 & 0.46 & 0.196 \\ 
        \botrule
        $\beta_4$ & -4    & -3   & -2    & 1     & 2     & 4     & 5  \\ 
        \midrule
        $X_4$ & 0.5 & 0.452 & 0.332 & 0.172 & 0.076 & 0.0 & 0.0 \\ 
        \botrule
        \end{tabular}
    \end{center}
    \end{table}
    In our simulation, we fix the drifts for the other queues to be 1, namely, $\beta_j = 1$ for $j \neq i$, and only vary the drift of the $i$th queue, $\beta_i$, and we obtain the proportion of time the corresponding queue $X_i$ remains at zero. 
    It is straightforward to see that larger drifts provide positive forces to drag queue paths travel away from zero more quickly. 
    \begin{figure}[h]
        \centering
        \subfigure{\includegraphics[width=0.46\textwidth]{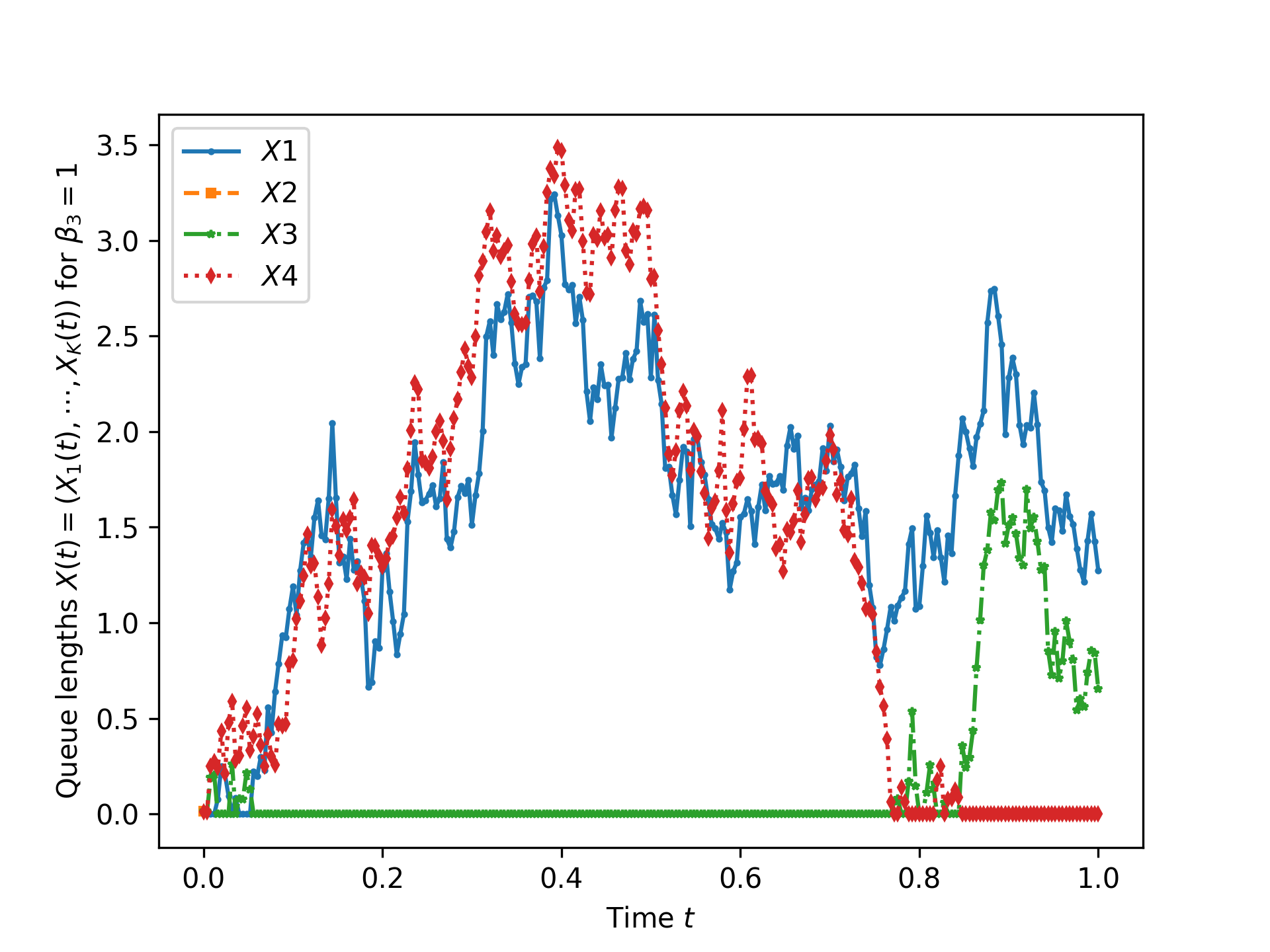}} 
        \subfigure{\includegraphics[width=0.46\textwidth]{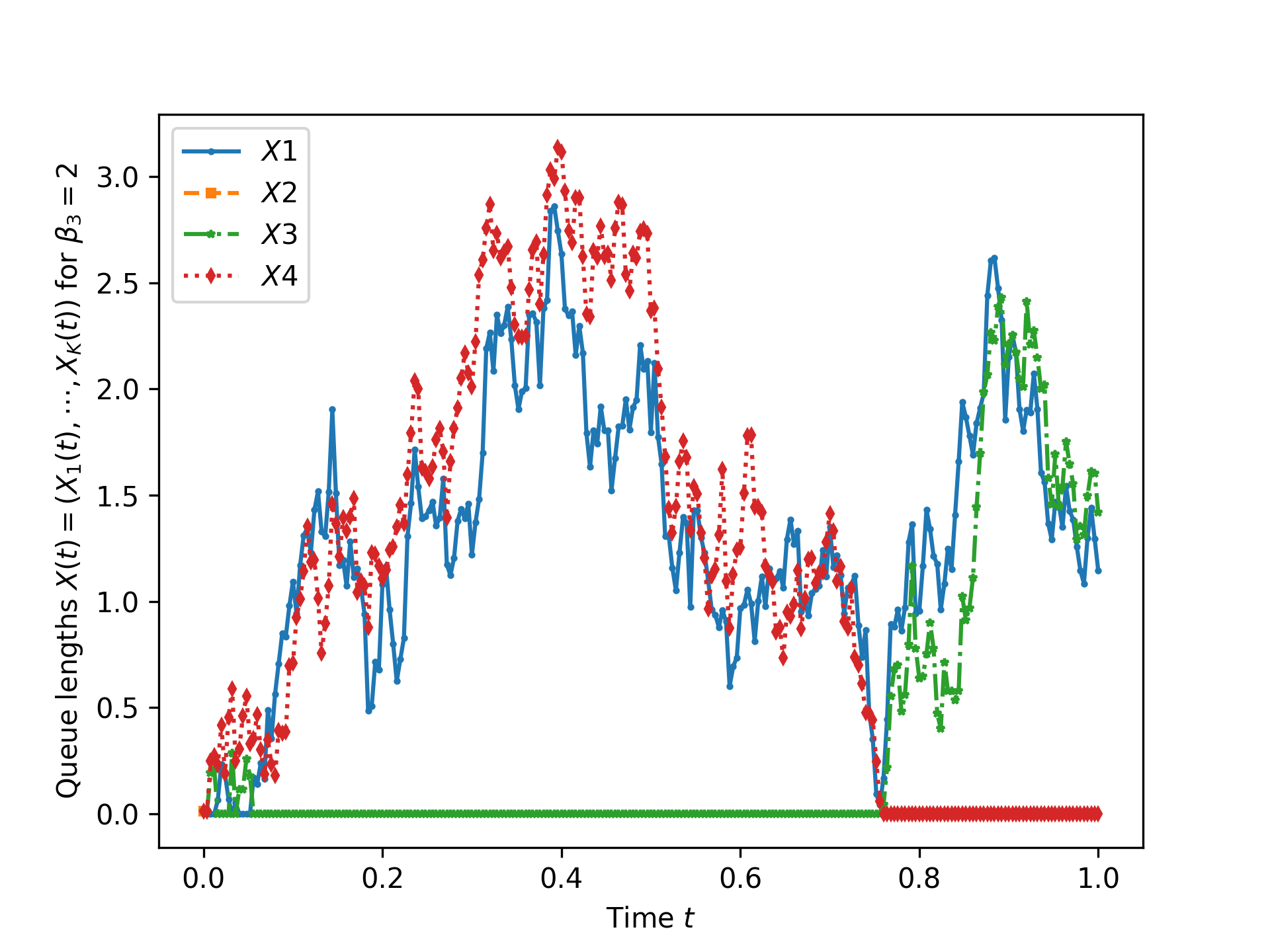}} 
        \subfigure{\includegraphics[width=0.46\textwidth]{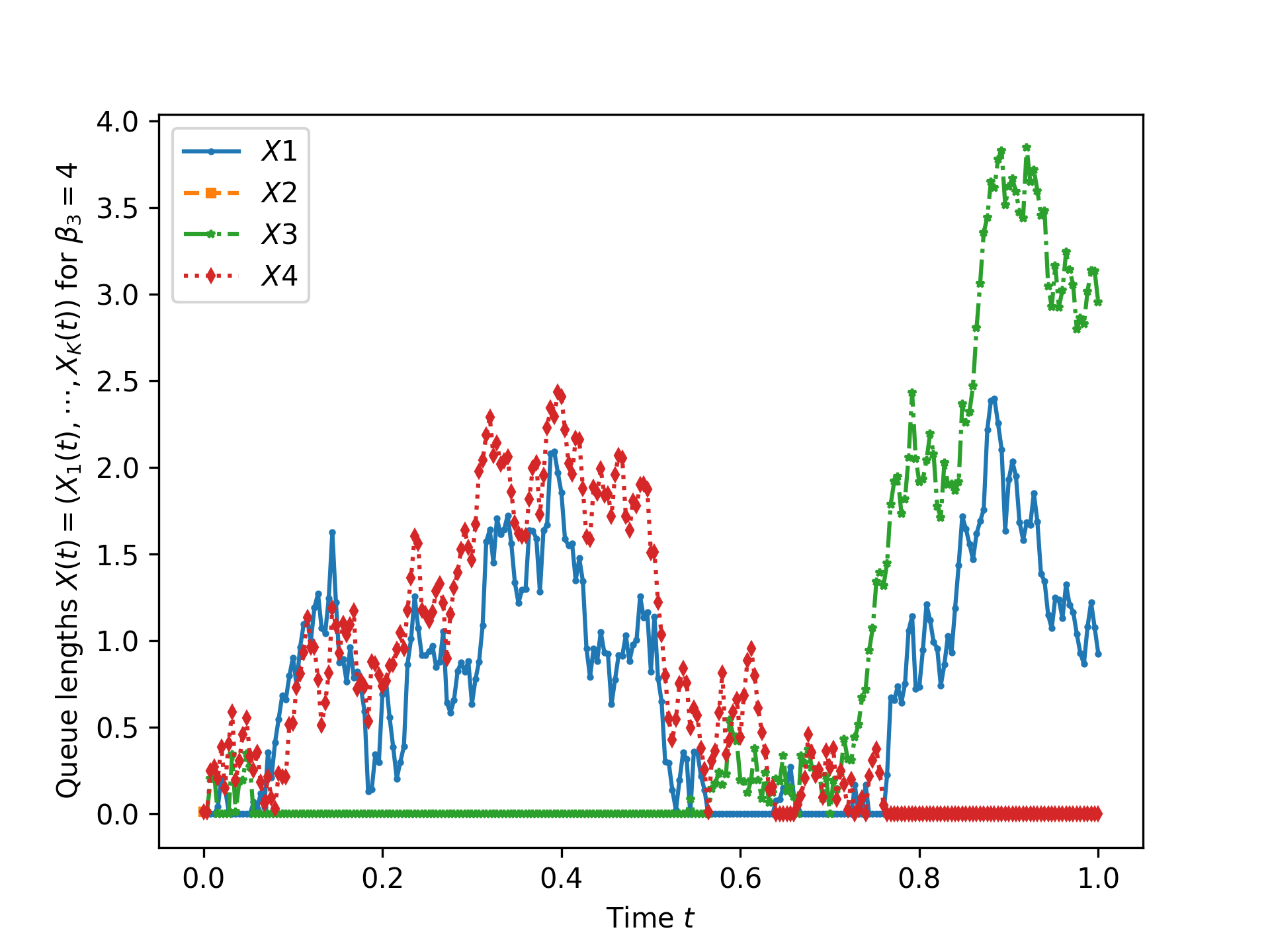}}
        \subfigure{\includegraphics[width=0.46\textwidth]{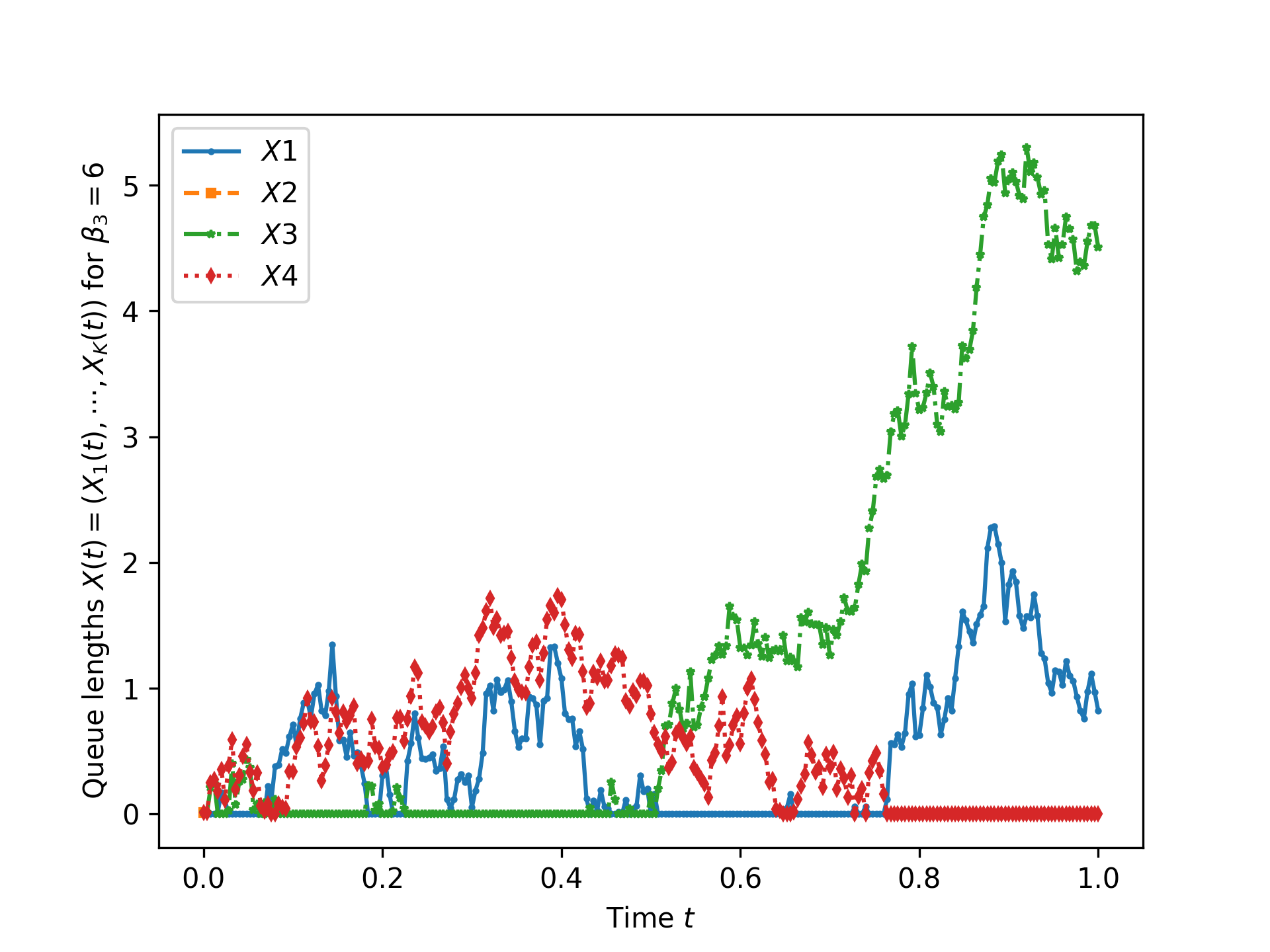}}
        \caption{Queue lengths in heavy traffic for different drift $\beta_3$ values}
        \label{fig: queue lengths for different drifts}
    \end{figure}
\end{example}



\section{Conclusion}\label{sec: conclusion}

In this article, we start by proposing the asymptotic framework of a sequence of multi-component matching queue systems with perishable components involving an explicit matching completion process and then establishing the heavy traffic limit under mild assumptions. 
We observe that the limiting process is characterized by a stochastic integral equation with a special coupling phenomenon, which can be called the coupled stochastic integral equation. To the best of our knowledge, this is relatively new compared with the conventional heavy-traffic limits, and its properties are of interest. 
We also construct the asymptotic Little's law for the matching queue system and further exhibit the convergence of cost functionals so that it provides the possibility to formulate control problems. 
We exhibit some numerical examples of the dynamics of coupling queueing systems to provide insights as category $K\to\infty$ and the stickiness depending on drift coefficients. 

There are numerous future works that can be done related to the matching queue models. 
In our follow-up paper \cite{xiewu2023control}, we impose each category queue with a buffer size, which could be controlled later to achieve an optimal profit situation under proper cost structures. The difficulty comes from establishing the asymptotic heavy-traffic approximation due to the twisted relationships. Recently, we had some interesting yet non-trivial ideas to prove the well-posedness of such potential coupled regulated limiting processes by studying the queues that reach the origin and boundaries, separately. Meanwhile, we also attempt to relax the assumptions by taking general distributed arrivals and abandonment distributions. 
It is also natural to extend the matching system as an interface provided to customers to generate distinct products, where we may consider a hypergraphical matching system along with some matching structures such that each node gathers multiple distinct components (cf. \cite{rahme2021stochastic}). However, an explicit expression for the matching completion should be established to characterize all the matches from each node, and the order of the long-run equilibrium in time $t\to\infty$ and in sequence $n\to\infty$ are of interest. We will consider these projects in further studies. 
\bmhead{Acknowledgments}

The author would like to acknowledge his advisor, Ananda Weerasinghe, for his guidance, patience, enthusiasm, and inspiration throughout the research and the writing of the paper. The author also would like to acknowledge Xin Liu and Ruoyu Wu for their suggestions on the content and structure of this paper.

\begin{appendices}

\setcounter{equation}{0}

\section{Proofs}
\label{appendix1}

\subsection{Proof of Lemma \ref{prove hat M is a martingale}}
\label{sec: prove hat M is a martingale}

\begin{proof}[Proof of Lemma \ref{prove hat M is a martingale}]
Observe that $I_i^n(t)$ is a stochastic process with continuous non-decreasing non-negative sample paths. 
We also know that $\{I_i^n(t) < x\}\in\bar{\mathcal{F}}_x^n$ for all $x\geq0$ and $t\geq0$, since to know $I_i^n(t)$, we need all the information of $Q_i^n(s)$ for $0\leq s\leq t$, which depends on $I_i^n(s)$ for all $0\leq s\leq t$ by \eqref{queue length}. 
Thus, to evaluate $I_i^n(t)<x$, it suffices to consider $\{N_i(u): 0\leq u\leq x\}$. This concludes that $I_i^n(t)$ is an $\bar{\mathcal{F}}_x^n$-stopping time for each $t\geq0$. By \eqref{queue length} and the non-negativity of $G_i^n$ in \eqref{abandonment process} and $R^n$ in \eqref{number of matches with G}, we have a crude inequality $Q_i^n(t) = Q_i^n(0) + A_i^n(t) - G_i^n(t) - R^n(t) \leq Q_i^n(0) + A_i^n(t)$. 
Using this inequality and since $Q_i^n(0)$ is deterministic, we further have
\begin{equation*}
\begin{aligned}
    E\left[\delta_i^n\int_0^tQ_i^n(s)ds\right] &\leq t\delta_i^n\left(Q_i^n(0) + E\left[A_i^n(t)\right]\right)
    =t\delta_i^n\left(Q_i^n(0) + \lambda_i^n t\right)<\infty,
\end{aligned}
\end{equation*}
and
\begin{equation*}
\begin{aligned}
    E\left[N_i\left(\delta_i^n\int_0^t Q_i^n(s)ds\right)\right] &\leq E\left[N_i\left(t\delta_i^n(Q_i^n(0) + A_i^n(t))\right)\right]
    =t\delta_i^n(Q_i^n(0) + \lambda_i^n t)<\infty. 
\end{aligned}
\end{equation*}
Since all the conditions of Lemma 3.2 in \cite{pang2007martingale} are fulfilled, we conclude that
\begin{equation*}
    N_i\left(\delta_i^n\int_0^t Q_i^n(s)ds\right) - \delta_i^n\int_0^tQ_i^n(s)ds,
\end{equation*}
is a square-integrable martingale with respect to $(\bar{\mathcal{F}}_{I_i^n}^n)$. Consequently, $\hat{M}_i^n$ is a square-integrable $\mathcal{F}_t^n$-martingale with quadratic variation process in \eqref{quadratic variation process} since the increments of arrival process $A_i^n(t+s) - A_i^n(t)$ for $s\geq0$ is independent of $Q_i^n(s)$ for $0\leq s\leq t$. 
\end{proof}

\subsection{Proof of Proposition \ref{SB for hat M}}
\label{sec: SB for hat M}

\begin{proof}[Proof of Proposition \ref{SB for hat M}]
    Since $\hat{M}_i^n$ is a martingale, by the Burkholder's inequality (see Theorem 45 in Protter \cite{protter2005stochastic}) and \eqref{quadratic variation process}, we have
    \begin{equation*}
        E\left[\|\hat{M}_i^n\|_T^2\right] \leq \Tilde{C} E\left[[\hat{M}_i^n, \hat{M}_i^n](T)\right] 
        = 
        \Tilde{C} E\left[\delta_i^n\int_0^T \bar{Q}_i^n(s)ds\right],
    \end{equation*}
    where $\Tilde{C}$ is some positive constant. 
    Since $A_i^n$ is a Poisson arrival process and $\lambda_i^n/n\to\lambda_0$ by \eqref{regime}, we further have
    \begin{equation*}
    \begin{aligned}
        E\left[\delta_i^n\int_0^T \bar{Q}_i^n(s)ds\right] 
        \leq T\delta_i^n \left(\bar{Q}_i^n(0) + \frac{1}{\sqrt{n}}E\left[\|\hat{A}_i^n\|_T\right] + \frac{\lambda_i^n}{n}T\right)
        \leq C_1(1+T^l), 
    \end{aligned}
    \end{equation*}
    where $C_1$ and $l\geq2$ are constants independent of $T$ and $n$. 
    This concludes \eqref{moment bound for hat M}. Consequently, by the Chebyshev's inequality, we have
    \begin{equation*}
        \lim_{a\to\infty}\limsup_{n\to\infty} P\left[\|\hat{M}_i^n\|_T^2>a\right] = 0.
    \end{equation*}
This completes the proof. 
\end{proof}

\subsection{Proof of Proposition \ref{moment bound and stochastical boundedness of hat Q}}
\label{sec: proof of moment bound and stochastical boundedness of hat Q}

\begin{proof}[Proof of Proposition \ref{moment bound and stochastical boundedness of hat Q}]
    Considering the martingale representation \eqref{rewritten diffusion scaled queue length}, we Assume 
    \begin{equation}
        B_T^n = \sum_{k=1}^K \left(\hat{Q}_k^n(0) + \|\hat{A}_k^n\|_T + \left\lvert\frac{\lambda_k^n - \lambda_0 n}{\sqrt{n}}\right\rvert T + \|\hat{M}_k^n\|_T\right). 
        \label{assume an upper bound for input}
    \end{equation}
    By \eqref{regime}, \eqref{moment condition on hat A}, and \eqref{moment bound for hat M}, we can represent $B_T^n := B_T^n(\omega)$ as a square-integrable random variable with the second moment bound $E\left[(B_T^n)^2\right]\leq C_2(1+T^b)$, where $C_2$ and $b\geq 2$ are constants independent of $T$ and $n$. 
    Next, we intend to find a moment bound for $\vect{\hat{Q}^n}$ as the following: 
    \begin{equation*}
    \begin{aligned}
        \sum_{k=1}^K \abs{\hat{Q}_k^n(t)} 
        \leq B_T^n + c_0 \int_0^t \sum_{k=1}^K \abs{\hat{Q}_k^n(s)} ds + K\abs{\hat{R}^n(t)}, 
    \end{aligned}
    \end{equation*}
    assuming $\sup_{1\leq k\leq K} (\delta_k^n)\leq c_0$ for some constant $c_0>0$ and for all $n>0$ since $\lim_{n\to\infty}\delta_i^n = \delta_i$. 
    
    Now it suffices to consider the last term on the right-hand side. 
    Notice that \eqref{rewritten hat R} suggests that for any $k\in\{1, \cdots, K\}$, 
    \begin{equation*}
    \begin{aligned}
        \hat{R}^n(t) 
        &\leq \left\lvert\hat{Q}_k^n(0) + \hat{A}_k^n(t) + \frac{\lambda_k^n - \lambda_0 n}{\sqrt{n}}t - \hat{M}_k^n(t) - \delta_k^n \int_0^t \hat{Q}_k^n(s)ds\right\rvert\\
        &\leq B_T^n + \sum_{k=1}^K \delta_k^n \int_0^t \abs{\hat{Q}_k^n(s)}ds. 
    \end{aligned}
    \end{equation*}
    The first inequality holds for all $k\in\{1, \cdots, K\}$ since the scalar-valued process $\hat{R}^n$ is defined to be the minimum value as described in \eqref{rewritten hat R}. Moreover, \eqref{rewritten hat R} also suggests an identical upper bound. 
    These implies an upper bound of $\abs{\hat{R}^n(t)}$, namely
    \begin{equation}
        \abs{\hat{R}^n(t)}
        \leq B_T^n + \sum_{k=1}^K \delta_k^n\int_0^t \abs{\hat{Q}_k^n(s)}ds.
        \label{upper bound of hat R}
    \end{equation}
    This together with previous inequalities of $\sum_{k=1}^K \abs{\hat{Q}_k^n(t)}$, we further have
    \begin{equation*}
    \begin{aligned}
        \sum_{k=1}^K \abs{\hat{Q}_k^n(t)}
        \leq (1+K)B_T^n + c_0(1+K) \int_0^t \sum_{k=1}^K  \abs{\hat{Q}_k^n(s)}ds.
    \end{aligned}
    \end{equation*}
    We apply the Gronwall's inequality to function $t\mapsto\sum_{k=1}^K \abs{\hat{Q}_k^n(t)}$ to obtain
    \begin{equation}
       \sum_{k=1}^K\abs{\hat{Q}_k^n(t)} \leq (1+K)B_T^n \exp{\left(c_0(1+K) T \right)},  
    \end{equation}
    which further yields
    \begin{equation}
        \|\vect{\hat{Q}^n}(t)\| = \left(\sum_{k=1}^K \abs{\hat{Q}_k^n(t)}^2\right)^{\frac{1}{2}} \leq \sum_{k=1}^K \abs{\hat{Q}_k^n(t)} \leq (1+K)B_T^n \exp{\left(c_0(1+K) T \right)}. 
    \end{equation}
    Consequently, we have the moment-bound result: 
    \begin{equation}
        E\left[\|\vect{\hat{Q}^n}\|_T^2\right] \leq (1+K)^2 C_2(1+T^b) \exp{\left(2c_0(1+K) T \right)}, 
        \label{moment bound for hat Q}
    \end{equation}
    which further implies \eqref{moment condition on hat Q}. 
    The stochastic boundedness follows by employing Chebyshev's inequality. This completes the proof. 
\end{proof}


\subsection{Proof of Theorem \ref{continuity of integral representation}}
\label{sec: continuity of integral representation}

\begin{proof}[Proof of Theorem \ref{continuity of integral representation}]
    For brevity, we intend to consider the case of $h(\vect{x}(t)) = (\delta_1 x_1(t), \delta_2 x_2(t), \cdots, \delta_K x_K(t))^\intercal$ for $t\geq0$, which is a special case of the integral term in the heavy traffic limit obtained in \eqref{limiting process (model with abandonment)}. 
    For fixed $\vect{y}(t)\in D^K[0, T]$, we define a functional $M: D^K[0, T]\to D^K[0, T]$ by
    \begin{equation}
        M(\vect{x}(t)) = \vect{y}(t) - \int_0^t h(\vect{x}(s))ds - R(t)\vect{I}, 
        \label{functional M}
    \end{equation}
    where $R(\cdot)$ is defined in \eqref{definition of R functional version}. 
    To demonstrate the existence of a unique solution, it suffices to show $M$ is a contraction mapping on $ D^K[0, T]$ embedded with the uniform topology. 
    Suppose there are two solutions to the integral representation \eqref{integral representation}, namely $\vect{x}^{(1)}(t)$ and $\vect{x}^{(2)}(t)$ for $t\geq0$. 
    Accordingly, we have
    \begin{equation}
        R^{(k)}(t) = \Psi(\vect{x}^{(k)}, \vect{y})(t)= \min_{1\leq j \leq K} \left\{y_j(t) - \int_0^t \delta_j x^{(k)}_j(s)ds\right\}, 
        \label{R1 and R2 definition}
    \end{equation}
    for $k = 1, 2$. 
    The functional $M$ defined in \eqref{functional M} suggests
    \begin{equation*}
    \begin{aligned}
        \|M(\vect{x}^{(1)}(t)) - M(\vect{x}^{(2)}(t)) \| 
        &\leq \sum_{k=1}^K \delta_k \int_0^t \left\lvert x_k^{(1)}(s) - x_k^{(2)}(s)\right\rvert ds + K\left\lvert R^{(1)}(t) - R^{(2)}(t) \right\rvert. 
    \end{aligned}
    \end{equation*}
    The complication comes from the second term $\abs{R^{(1)}(t) - R^{(2)}(t)}$. To find an upper bound, we assume there exists some $l\in[1, K]$ depends on $t$ so that it achieves the minimum in $R^{(2)}(t)$. By \eqref{R1 and R2 definition}, we have for $t\in[0, T]$,
    \begin{equation*}
    \begin{aligned}
        R^{(1)}(t) - R^{(2)}(t) 
        &\leq y_l(t) - \int_0^t \delta_l x^{(1)}_j(s)ds - \left(y_l(t) - \int_0^t \delta_l x^{(2)}_l(s)ds\right)\\
        &\leq \left\lvert \int_0^t \delta_l\left(x^{(1)}_l(s) - x^{(2)}_l(s)\right)ds\right\rvert \\
        &\leq \sup_{1\leq j\leq K} (\delta_j) \cdot \int_0^t \sum_{j=1}^K \left\lvert x^{(1)}_j(s) - x^{(2)}_j(s)\right\rvert ds.
    \end{aligned}
    \end{equation*}
    Similarly, we can obtain an identical upper bound for $R^{(2)}(t) - R^{(1)}(t)$. 
    Thus, 
    \begin{equation}
        \left\lvert R^{(1)}(t) - R^{(2)}(t)\right\rvert \leq \sup_{1\leq k\leq K} (\delta_k) \cdot \int_0^t \sum_{k=1}^K \left\lvert x^{(1)}_k(s) - x^{(2)}_k(s)\right\rvert ds.
        \label{upper bound for abs{Z1-Z2}}
    \end{equation}
    Therefore, we have
    \begin{equation*}
    \begin{aligned}
        \|M(\vect{x}^{(1)}(t))- M(\vect{x}^{(2)}(t))\| 
        &\leq (1+K) \sup_{1\leq k \leq K} (\delta_k) \int_0^t\sum_{k=1}^K \left\lvert x^{(1)}_k(s) - x^{(2)}_k(s)\right\rvert ds\\
        &\leq \epsilon(T) \|x^{(1)}(t) - x^{(2)}(t)\|_T, 
    \end{aligned}
    \end{equation*}
    where $\epsilon(T) = (1+K)\sqrt{K}\left(\sup_{1\leq k \leq K} (\delta_k)\right)T$. 
    This yields
    \begin{equation*}
        \|M(\vect{x}^{(1)}(t))- M(\vect{x}^{(2)}(t))\|_T \leq \epsilon(T) \|x^{(1)}(t) - x^{(2)}(t)\|_T, 
    \end{equation*}
    One may pick $T_1>0$ such that $\epsilon(T_1)<1$, and then the functional $M$ formulates a contraction mapping for $t\in[0, T_1]$ on $ D^K[0, T_1]$ with uniform topology, which leads to the existence of a unique solution to \eqref{integral representation} by the Banach fixed-point theorem. 
    If we partition the time interval $[0, T]$ into several length $T_1$ subintervals, we can apply the above arguments on each one of those length $T_1$ subintervals to obtain a unique solution for all $t\in[0, T]$. These guarantees a unique solution $\vect{x}\in D^K[0, T]$ to the fixed point problem $M(\vect{x}) = \vect{x}$. 
    
    The continuity of $f$ can be deduced by considering $\|f(\vect{y}(t_n)) - f(\vect{y}(t))\|$. Analogous to previous discussions, we end up with the following inequality: 
    \begin{equation*}
    \begin{aligned}
        \|\vect{x}(t_n) - \vect{x}(t)\| &= \|f(\vect{y}(t_n)) - f(\vect{y}(t))\|\\
        &\leq (1+K)\sqrt{K} \|\vect{y}(t_n) - \vect{y}(t)\| + (1+K)\sqrt{K} \sup_{1\leq k\leq K} (\delta_k) \int_t^{t_n} \|\vect{x}(s)\|ds. 
    \end{aligned}
    \end{equation*}
    If we impose the boundedness for $\vect{x}(\cdot)$, $\vect{x}$ is continuous if $\vect{y}$ is continuous. 
\end{proof}


\subsection{Proof of Corollary \ref{L^p extension}}
\label{sec: L^p extension}
\begin{proof}[Proof of Corollary \ref{L^p extension}]
The proof of this extension relies on verifying the uniform integrability of a proper integrand. 
Since \eqref{initial state limit}, \eqref{regime}, and \eqref{weak convergence of hat A}, we have  $\vect{\xi^n}\Rightarrow\vect{\xi}$ in $ D^K[0, T]$ as $n\to\infty$. 
By Skorokhod's representation theorem, we can simply assume that $\vect{\xi^n}$ converges to $\vect{\xi}$ a.s. in some special probability space. 
For given $\vect{\xi^n}$ and $\vect{\xi}$ in conjunction with Theorem \ref{continuity of integral representation}, we obtain $\vect{\hat{Q}^n}$ and $\vect{X}$ associated with the corresponding input processes $\vect{\xi^n}$ and $\vect{\xi}$ so that they solve \eqref{integral representation}, respectively. Therefore, we have 
\begin{equation}
\begin{aligned}
    \sum_{j=1}^K \abs{\hat{Q}_j^n(t) - X_j(t)} 
    \leq \sum_{j=1}^K \abs{\xi_j^n(t) - \xi_j(t)} + \int_0^t \sum_{j=1}^K \abs{\delta_j^n\hat{Q}_j^n(s) - \delta_jX_j(s)}ds + K\abs{\hat{R}^n(t) - R(t)}. 
\end{aligned}
\label{upper bound of hat Q - X}
\end{equation}
To find an upper bound, the difficulty also comes from the last term $\abs{\hat{R}^n(t) - R(t)}$. To this end, it suffices to find an upper bound for $\abs{\hat{R}^n(\cdot) - \hat{R}(\cdot)}$. Consider two differences without absolute value separately. 
We assume that there exist indices $l_1$ and $l_2$ depend on $t$ such that the minimum entry in $\hat{R}^n(t)$ is attained at $l_1$ and the minimum entry in $R(t)$ is attained at $l_2$.  
Hence, 
\begin{equation*}
\begin{aligned}
    \hat{R}^n(t) - R(t) &= \min_{1\leq k\leq K}\left\{\xi_k^n(t) - \int_0^t \delta_k^n\hat{Q}_k^n(s)ds\right\} - \min_{1\leq k\leq K} \left\{\xi_k(t) - \int_0^t \delta_k X_k(s)ds\right\} \\
    &\leq \xi_{l_2}^n(t) - \int_0^t \delta_{l_2}^n\hat{Q}_{l_2}^n(s)ds - \left(\xi_{l_2}(t) - \int_0^t \delta_{l_2} X_{l_2}(s)ds\right) \\
    &\leq \abs{\xi_{l_2}^n(t) - \xi_{l_2}(t)} + \int_0^t \abs{\delta_{l_2}^n\hat{Q}_{l_2}^n(s) - \delta_{l_2} X_{l_2}(s)}ds \\
    &\leq \sum_{j=1}^K \abs{\xi_{j}^n(t) - \xi_{j}(t)} + \int_0^t \sum_{j=1}^K\abs{\delta_{j}^n\hat{Q}_{j}^n(s) - \delta_{j} X_{j}(s)}ds. 
\end{aligned}
\end{equation*}
Notice that the first inequality holds since $\hat{R}^n(t) \leq \xi_k^n(t) - \int_0^t \delta_k^n\hat{Q}_k^n(s)ds$ for any $k\in\{1, \cdots, K\}$ and $t\geq0$. 
Similarly, we can obtain an upper bound for $R(t) - \hat{R}^n(t)$. 
Consequently, we have the following upper bound: 
\begin{equation}
    \abs{\hat{R}^n(t) - R(t)} \leq \sum_{j=1}^K \abs{\xi_{j}^n(t) - \xi_{j}(t)} + \int_0^t \sum_{j=1}^K\abs{\delta_{j}^n\hat{Q}_{j}^n(s) - \delta_{j} X_{j}(s)}ds. 
\end{equation}
This fact and \eqref{upper bound of hat Q - X} suggest that
\begin{equation*}
\begin{aligned}
    &\|\vect{\hat{Q}^n}(t) - \vect{X}(t)\| \\
    &\leq (1+K)\left(\sum_{j=1}^K \abs{\xi_{j}^n(t) - \xi_{j}(t)} + \int_0^t \sum_{j=1}^K\abs{\delta_{j}^n\hat{Q}_{j}^n(s) - \delta_{j} X_{j}(s)}ds\right) \\
    &\leq (1+K)\sqrt{K} \left[\left(\sum_{j=1}^K \abs{\xi_{j}^n(t) - \xi_{j}(t)}^2\right)^{\frac{1}{2}} + \int_0^t \left(\sum_{j=1}^K\abs{\delta_{j}^n\hat{Q}_{j}^n(s) - \delta_{j} X_{j}(s)}^2\right)^{\frac{1}{2}}ds\right] \\
    &\leq (1+K)\sqrt{K} \left(\|\vect{\xi^n}(t) - \vect{\xi}(t)\| + C_0 \int_0^t \|\vect{\hat{Q}^n}(s) - \vect{X}(s)\| ds\right), 
\end{aligned}
\end{equation*}
where we assume $\left(\sup_{1\leq k\leq K} \delta_k^n\right) \vee \left(\sup_{1\leq k\leq K} \delta_k\right) \leq C_0$ for some $C_0$ positive constant. 
Using the Gronwall's inequality, we obtain
\begin{equation}
    \|\vect{\hat{Q}^n}(t) - \vect{X}(t)\| \leq (1+K)\sqrt{K} \|\vect{\xi^n}- \vect{\xi}\|_T e^{(1+K)\sqrt{K}C_0 t}. 
    \label{upper bound for abs{hat Q - X}}
\end{equation}

Now, if we have the convergence of the right-hand side of \eqref{upper bound for abs{hat Q - X}}, it is straightforward to show the convergence of the left-hand side term. 
Notice that we have assumed almost sure convergence of $\vect{\xi^n}$, which further yields $\|\vect{\xi^n} - \vect{\xi}\|_T \to 0$ in probability as $n\to\infty$. 
We intend to show the convergence also holds in $L^p[0, T]$ for $1\leq p<2l$ where $l\geq1$ is any constant. That is the convergence holds for any $p\geq 1$. 
Here since we need to find higher moment bounds for appropriate processes in the proof, we tend to present constant $l$ for generality. 
Then, Vitali's convergence theorem suggests that if the $p$th order integrand is uniformly integrable and in conjunction with convergence in probability, it is straightforward to conclude the convergence in $L^p[0, T]$. 

We are left to show the uniform integrability. 
It is trivial that 
\begin{equation}
    E\left[\|\vect{\xi^n} - \vect{\xi}\|_T^{2l}\right] \leq c E\left[\|\vect{\xi^n}\|_T^{2l} + \|\vect{\xi}\|_T^{2l}\right], 
    \label{eq: E[ norm(vect(xi^n) - vect(xi)]}
\end{equation}
where $c>0$ is a generic constant, and we intend to find a moment bound for those two terms separately. 
Since the moment bound of the second term can be derived by the moment bound of the first term with the help of Fatou's lemma, it suffices to consider $E\left[\|\vect{\xi^n}\|_T^{2l}\right]$. 
\eqref{regime} and \eqref{moment condition on hat A} suggest
\begin{equation*}
\begin{aligned}
    E\left[\|\vect{\xi^n}\|_T^{2l}\right] 
    \leq c \left(1 + T^{2l} + E\left[\|\vect{\hat{A}^n}\|_T^{2l}\right] + E\left[\|\vect{\hat{M}^n}\|_T^{2l}\right]\right), 
\end{aligned}
\end{equation*}
where $c>0$ is a generic constant depends on $K$. 
Let $e=\{e(t):= t, t\geq0\}$ be the identity map. 
Since the centered and scaled arrival processes $\{\hat{A}_j^n\}_{1\leq j\leq K}$ are independent Poisson processes as assumed in Assumption 2, and $A_j^n - \lambda_j^n e$ is a $(\mathcal{F}_t^n)_{t\geq0}$ adapted martingale for each $j\in\{1, \cdots, K\}$, the Burkholder's inequality (see \cite{pang2007martingale}) renders
\begin{equation*}
\begin{aligned}
    E\left[\|\hat{A}_j^n\|_T^{2l}\right] &= \frac{1}{n^{l}}E\left[\|A_j^n - \lambda_j^n e\|_T^{2l}\right] \leq \frac{1}{n^l} E\left[\left([A_j^n - \lambda_j^n e, A_j^n - \lambda_j^n e](T)\right)^l\right]. 
\end{aligned}
\end{equation*}
The quadratic variation of compensated Poisson process implies $[A_j^n - \lambda_j^n e, A_j^n - \lambda_j^n e](T) = A_j^n(T)$ and $E\left[(A_j^n(T))^l\right]\leq c(\lambda_j^n T)^l$. As a consequence, 
\begin{equation}
    \sup_{n\geq1}E\left[\|\hat{A}_j^n\|_T^{2l}\right] \leq c T^l, 
    \label{higher order moment bound for hat A}
\end{equation}
where $c>0$ is a generic constant independent of $T$ and $n$. 
Similarly, since $\hat{M}_j^n$ is also a $(\mathcal{F}_t^n)_{t\geq0}$-martingale for each $j\in\{1, \cdots, K\}$ and analogous to the proof of Proposition \ref{SB for hat M}, the Burkholder's inequality yields
\begin{equation*}
    E\left[\|\hat{M}_j^n\|_T^{2l}\right]\leq c E\left[\left([\hat{M}_j^n, \hat{M}_j^n](T)\right)^l\right], 
\end{equation*}
where $c>0$ is a generic constant. 
Hence, since $Q_j^n(0)$ is deterministic and using \eqref{quadratic variation process}, a crude inequality $Q_j^n(s)\leq Q_j^n(0) + A_j^n(s)$ implies
\begin{equation*}
\begin{aligned}
    E\left[\|\hat{M}_j^n\|_T^{2l}\right] &\leq \frac{c}{n^l} E\left[\left(N_j\left(\delta_j^n\int_0^T Q_j^n(s)ds\right)\right)^l\right] \\
    &\leq \frac{c}{n^l} E\left[\left(N_i\left(\delta_i^n T (Q_j^n(0) + A_j^n(T))\right)\right)^l\right] \\
    &= \frac{c}{n^l} E\left[ E\left[\left(N_i(\delta_i^n T (Q_j^n(0) + A_j^n(T)))\right)^l \vert Q_j^n(0) + A_j^n(T)\right]\right] \\
    &\leq \frac{c}{n^l} T^l \left((Q_j^n(0))^l + E\left[(A_j^n(T))^l\right]\right) \\
    &\leq c T^l \left(1 + T^l\right),  
\end{aligned}
\end{equation*}
where $c>0$ is a generic constant. 
Therefore, we obtain the $(2l)$th moment bound condition
\begin{equation}
    \sup_{n\geq1} E\left[\|\vect{\xi^n}\|_T^{2l}\right] \leq c(1+T^d), 
    \label{higher order moment bound for input}
\end{equation}
where $c>0$ is a generic constant, and $c$ and $d\geq 2l\geq 2$ are both constants independent of $T$ and $n$. 
Hence, using \eqref{eq: E[ norm(vect(xi^n) - vect(xi)]}, we have
\begin{equation}
   \sup_{n\geq1} E\left[\|\vect{\xi^n} - \vect{\xi}\|_T^{2l}\right] \leq c(1+T^d), 
\end{equation}
where $c>0$ is a generic constant and $c$ and $d\geq 2l \geq 2$ are independent of $T$ and $n$. 
This implies the uniform integrability of $\|\vect{\xi^n} - \vect{\xi}\|_T^p$ for $1\leq p < 2l$. 
As a consequence, $E\left[\|\vect{\xi^n} - \vect{\xi}\|_T^p\right] \to 0$ as $n\to\infty$ on a special probability space. 
Using \eqref{upper bound for abs{hat Q - X}}, we further obtain $E\left[\|\vect{\hat{Q}^n} - \vect{X}\|_T^p\right] \to 0$. 
This completes the proof. 
\end{proof}
\subsection{Proof of Proposition \ref{Theorem semimartingale property}}
\label{sec: Proof of semimartingale property}

\begin{proof}[Proof of Proposition \ref{Theorem semimartingale property}]
To avoid redundant algebraic manipulations and for brevity, we intend to demonstrate the element indexed by $i=1$ with $K=4$ case, and other elements can be obtained by following the same fashion. 

To show the coupled process is a semimartingale, it suffices to prove that each component admits a semimartingale decomposition. 
The limiting processes in \eqref{limiting process (non-abandonment)} can be rewritten as
\begin{equation}
    X_i(t) = \xi_i(t) - \min\{\xi_1(t), \xi_2(t), \xi_3(t), \xi_4(t)\}, 
    \label{4 component processes (non-abandonment)}
\end{equation}
for $i=1, 2, 3, 4$ and $t\geq 0$, where 
\begin{equation}
    \xi_i(t) = X_i(0) + \beta_i t + \sigma_i W_i(t), 
    \label{input for limiting process (non-abandonmnet)}
\end{equation}
for each $i$ and $t\geq 0$. 
Consider the case of $i=1$. \eqref{4 component processes (non-abandonment)} further suggests that for $t\geq 0$, 
\begin{equation}
\begin{aligned}
    X_1(t) 
    &= \xi_1(t) + \max\{-\xi_1(t), \max\{-\xi_2(t), \eta_3(t)\}\}, 
\end{aligned}
\label{iteratively X_1 (non-abandonment)}
\end{equation}
where $\eta_3(t):= \max\{-\xi_3(t), -\xi_4(t)\}$. 
Observe that $\eta_3(t)$ can be further rewritten as
\begin{equation*}
    \eta_3(t) = -\xi_4(t) + (\xi_4(t) - \xi_3(t))^+. 
\end{equation*}
By utilizing Tanaka's formula (see Section 7.3 in \cite{chung1990introduction}), we apply It\^o's lemma to the function $f(x) = x^+$ for $x\in\mathbb{R}$ and obtain
\begin{equation*}
\begin{aligned}
    \eta_3(t) &= \max\{-\xi_3(t), -\xi_4(t)\} = -\xi_4(t) + (\xi_4(t)-\xi_3(t))^+ \\
    &= - X_4(0) - \sigma_4 W_4(t) - \beta_4 t + (X_4(0)-X_3(0))^+ \\
    &\quad + \sqrt{\sigma_3^2 + \sigma_4^2} \int_0^t \mathbbm{1}_{[Y_3(s)>0]}dB_{34}(s) + (\beta_4-\beta_3)\int_0^t \mathbbm{1}_{[Y_3(s)>0]}ds + \frac{1}{2}L_t^{(1)}, 
\end{aligned}
\end{equation*}
where
\begin{equation}
    Y_3(t) := \xi_4(t) - \xi_3(t) = (X_4(0)- X_3(0)) + \sqrt{\sigma_3^2 + \sigma_4^2} B_{34}(t) + (\beta_4 - \beta_3)t, 
\end{equation}
and $L_t^{(1)}$ is the local time process for $Y_3(t)$ at the origin, which increases only at time $t$ when $\xi_3(t) = \xi_4(t)$. 
Here, $B_{34}(\cdot)$ is a Brownian motion depends on two independent standard Brownian motions $W_3$ and $W_4$ obtained in Proposition \ref{weak convergence theorem (non-abandonment)}. 
Similarly, let $\eta_2(t):= \max\{-\xi_2(t), \eta_3(t)\}$ in \eqref{iteratively X_1 (non-abandonment)} and Tanaka's formula again yields a similar expression with a new local time process. 
Following the same fashion, one can move on to the last layer $\max\{-\xi_1, \eta_2\}$, and iteratively, we can obtain a semimartingale decomposition. 
The decomposition for general $K \geq 2$ can be done similarly. 
\end{proof}

\subsection{Proof of Theorem \ref{little's law}}
\label{sec: Proof of little's law}

First, we prove some results of interest, which will play an important role in the proof of Little's law. Then, we present the proof of Theorem \ref{little's law}. 

\begin{corollary}
    Let $T>0$ and for each $i\in\{1, \cdots, K\}$, we have that $\hat{G}_i^n$ is stochastically bounded and $\hat{G}_i^n(\cdot)$ converges weakly to $\delta_i \int_0^{\cdot} X_i(s)ds$ in $ D[0, T]$ as $n\to\infty$. 
    \label{SB for hat G and weak convergence of hat G}
\end{corollary}

\begin{proof}
We prove the result for the $i$th queue, and other queues can be proved in a very similar approach. 
By \eqref{hat M martingale}, we have
\begin{equation}
    \hat{G}_i^n(t) = \hat{M}_i^n(t) + \delta_i^n\int_0^t \hat{Q}_i^n(s)ds, 
\end{equation}
for $t\geq 0$. 
Using Proposition \ref{SB for hat M} and \ref{moment bound and stochastical boundedness of hat Q}, we derive the second moment bound result: 
\begin{equation*}
\begin{aligned}
    E\left[\|\hat{G}_i^n\|_T^2\right] 
    &\leq 2\left(E\left[\|\hat{M}_i^n\|_T^2\right] + (\delta_i^n)^2 T^2 E\left[\|\hat{Q}_i^n\|_T^2\right]\right) \\
    &\leq 2C_1(1+T^l) + 2C^2 K(1+K)^2 C_2 T^2 (1+T^b) \exp{(2C(1+K)T)}, 
\end{aligned}
\end{equation*}
where $C$, $C_1$, $C_2$, $l$, and $b$ are constants independent of $T$ and $n$ as described in \eqref{moment bound for hat M} and \eqref{moment bound for hat Q}. 
Therefore, using the Chebyshev's inequality, we have $\lim_{a\to\infty}\limsup_{n\to\infty} P\left[\|\hat{G}_i^n\|_T>a\right] = 0$. 

Next, we show the weak convergence. 
Since $\hat{M}_i^n$ is a $\mathcal{F}^n$-martingale by Lemma \ref{prove hat M is a martingale} and as in the proof of Proposition \ref{SB for hat M}, using the Burkholder's inequality, we have
\begin{equation*}
    E\left[\|\hat{M}_i^n\|_T^2\right] \leq C E\left[\delta_i^n \int_0^T \bar{Q}_i^n(s)ds\right]
    \leq \frac{C\delta_i^n T}{\sqrt{n}}\left(E\left[\|\hat{Q}_i^n\|_T^2\right]\right)^{\frac{1}{2}}. 
\end{equation*}
Using the moment bound result of $\hat{Q}_i^n$ in Proposition \ref{moment bound and stochastical boundedness of hat Q} and we have assumed that $\lim_{n\to\infty}\delta_i^n = \delta_i>0$, we obtain $E\left[\|\hat{M}_i^n\|_T^2\right]$ converges to zero as $n\to\infty$. By Chebyshev's inequality, we further have $\|\hat{M}_i^n\|_T$ converges to zero in probability as $n\to\infty$. 
Since Theorem \ref{weak convergence theorem} implies the weak convergence of $\hat{Q}_i^n$ in $ D[0, T]$, the continuity of integral mappings further suggests that $\delta_i^n\int_0^{\cdot} \hat{Q}_i^n(s)ds$ converges weakly to $\delta_i\int_0^{\cdot} X_i(s) ds$ in $ D[0, T]$. As a consequence, $\hat{G}_i^n(\cdot)$ converges weakly to $\delta_i\int_0^{\cdot} X_i(s) ds$ in $ D[0, T]$ as $n\to\infty$. 
\end{proof}

Now, with the facts obtained above, we are ready to see some crucial properties for the virtual waiting time processes introduced in \eqref{virtual waiting time}. 

\begin{proposition}
    Under the assumptions of Theorem \ref{weak convergence theorem} and for each $i\in\{1, \cdots, K\}$, we have that $\hat{V}_i^n$ is stochastically bounded and consequently, $\|V_i^n\|_T\to0$ in probability as $n\to\infty$. 
    \label{SB for hat V}
\end{proposition}

\begin{proof}
This argument is similar to the idea of proving Proposition 4.4 in \cite{dai2010customer}. 
Let $M>0$ be arbitrary. If $0<M<\hat{V}_i^n(t)$ for some $t\in[0, T]$, then we have $V_i^n(t)>\frac{M}{\sqrt{n}}$, which suggests that the queue length of category $i$ at time $t+\frac{M}{\sqrt{n}}$ is not empty and 
\begin{equation}
    Q_i^n\left(t+\frac{M}{\sqrt{n}}\right) \geq A_i^n\left(t+\frac{M}{\sqrt{n}}\right)-A_i^n(t) - \mathring{G}_i^n\left(t+\frac{M}{\sqrt{n}}\right), 
\end{equation}
where $\mathring{G}_i^n\left(t, t+\frac{M}{\sqrt{n}}\right)$ represents the amount of abandoned components from the $i$th queue for those arrivals during $[t, t+\frac{M}{\sqrt{n}})$. It counts those abandoned items that arrived after time $t$ and abandoned before time $t+\frac{M}{\sqrt{n}}$. 
We further observe that the number of abandoned components among those arrivals is less than the number of abandoned components by time $t+\frac{M}{\sqrt{n}}$, namely $0\leq\mathring{G}_i^n\left(t+\frac{M}{\sqrt{n}}\right)\leq G_i^n\left(t+\frac{M}{\sqrt{n}}\right)$, since those arrivals before time $t$ may abandon the system during the time interval $[t, t+\frac{M}{\sqrt{n}})$. 
Therefore, together with a simple computation, we have a diffusion-scaled inequality: 
\begin{equation}
    \hat{Q}_i^n\left(t+\frac{M}{\sqrt{n}}\right) \geq \hat{A}_i^n\left(t+\frac{M}{\sqrt{n}}\right) - \hat{A}_i^n(t) + \frac{\lambda_i^n}{n} M - \hat{G}_i^n\left(t+\frac{M}{\sqrt{n}}\right). 
\end{equation}

Let $0<\delta<1$. 
Since we assumed $\lambda_i^n/n\to \lambda_0$ as $n\to\infty$ by \eqref{regime}, we can find an $\alpha>0$ and $N\geq 1$ so that for any $n\geq N$, we have $0<\frac{M}{\sqrt{n}}<\delta$ and $\frac{\lambda_i^n}{n}>3\alpha > 0$ hold. 
Hence, for any $n\geq N$, we have the following inclusion: 
\begin{equation}
    \left[\|\hat{V}_i^n\|_T > M\right] \subseteq \left[\left\lvert \hat{Q}_i^n\left(t+\frac{M}{\sqrt{n}}\right)\right\rvert + \left\lvert\hat{A}_i^n\left(t+\frac{M}{\sqrt{n}}\right) - \hat{A}_i^n(t)\right\rvert + \left\lvert\hat{G}_i^n\left(t+\frac{M}{\sqrt{n}}\right)\right\rvert> 3\alpha M\right]. 
\end{equation}
Therefore, 
\begin{equation*}
\begin{aligned}
    P\left[\|\hat{V}_i^n\|_T > M\right] 
    &\leq P\left[\left\lvert\hat{Q}_i^n\left(t+\frac{M}{\sqrt{n}}\right)\right\rvert >\alpha M\right] + P\left[\left\lvert\hat{A}_i^n\left(t+\frac{M}{\sqrt{n}}\right) - \hat{A}_i^n(t)\right\rvert > \alpha M\right] \\
    &\quad + P\left[\left\lvert\hat{G}_i^n\left(t+\frac{M}{\sqrt{n}}\right)\right\rvert > \alpha M\right] \\
    &\leq P\left[\|\hat{Q}_i^n\|_{T+1} >\alpha M\right] + P\left[\|\hat{A}_i^n(t) - \hat{A}_i^n(s)\|_{0<s<t<(s+\delta)\wedge(T+1)} > \alpha M\right] \\
    &\quad + P\left[\|\hat{G}_i^n\|_{T+1}> \alpha M\right]. 
\end{aligned}
\end{equation*}
Since the weak convergence of $\hat{A}_i^n$ in \eqref{weak convergence of hat A}, we can further obtain the tightness of $\hat{A}_i^n$ and it also satisfies $\lim_{\delta\to0}\limsup_{n\to\infty} P\left[\omega(\hat{A}_i^n, \delta, T) > \epsilon\right]=0$. 
Using this fact and together with Proposition \ref{moment bound and stochastical boundedness of hat Q} and Corollary \ref{SB for hat G and weak convergence of hat G}, we obtain stochastic boundedness of $\hat{V}_i^n$. Consequently, $\lim_{n\to\infty} \|V_i^n\|_T = 0$ in probability. 
\end{proof}

Now, we are ready to prove Theorem \ref{little's law}. 

\begin{proof}[Proof of Theorem \ref{little's law}]
We will prove the result in terms of category $i$ and the cases for other categories remain identical. 
Consider the state of the $i$th queue at time $t+V_i^n(t)$ for any $t\in[0, T]$. We observe that the queue length at time $t+V_i^n(t)$ equals the number of arrivals during $[t, t+V_i^n(t))$ minus the number of abandoned items among those arrivals and this relation can be characterized by the following equality: 
\begin{equation}
    Q_i^n(t+V_i^n(t)) = A_i^n(t+V_i^n(t)) - A_i^n(t) - \mathring{G}_i^n(t, t+V_i^n(t)), 
    \label{equality at t + V_i^n(t)}
\end{equation}
where $\mathring{G}_i^n(t, t+V_i^n(t))$ represents the amount of abandoned components who arrived after time $t$ and abandoned before $t+V_i^n(t)$. 
We scale both sides of \eqref{equality at t + V_i^n(t)} by $1/\sqrt{n}$ and with a simple algebraic manipulation, we can obtain
\begin{equation}
    \hat{Q}_i^n(t+V_i^n(t)) = \hat{A}_i^n(t+V_i^n(t)) - \hat{A}_i^n(t) + \frac{\lambda_i^n}{n}\hat{V}_i^n(t) - \hat{\mathring{G}}_i^n(t, t+V_i^n(t)),  
\end{equation}
where the diffusion-scaled $\hat{Q}_i^n$ and $\hat{A}_i^n$ are as defined in \eqref{diffusion scales}, and 
\begin{equation}
    \hat{\mathring{G}}_i^n(t, t+V_i^n(t)) := \frac{1}{\sqrt{n}} \mathring{G}_i^n(t, t+V_i^n(t)). 
\end{equation}

Consider the last term $\hat{\mathring{G}}_i^n$, we observe that 
\begin{equation}
    0\leq \hat{\mathring{G}}_i^n(t, t+V_i^n(t)) \leq \frac{1}{\sqrt{n}}\left(G_i^n(t+V_i^n(t)) - G_i^n(t)\right) = \hat{G}_i^n(t+V_i^n(t)) - \hat{G}_i^n(t), 
\end{equation}
since those who arrived before time $t$ may abandon right after time $t$ and still before time $t+V_i^n(t)$, and those abandoned items are not counted in $\mathring{G}_i^n(t, t+V_i^n(t))$. 
With this observation, we have
\begin{equation*}
\begin{aligned}
    &\|\hat{Q}_i^n(t+V_i^n(t)) - \lambda_0 \hat{V}_i^n(t)\|_T \\
    &\leq \|\hat{A}_i^n(t+V_i^n(t)) - \hat{A}_i^n(t)\|_T + \left\lvert\frac{\lambda_i^n}{n} - \lambda_0\right\rvert \|\hat{V}_i^n\|_T + \|\hat{G}_i^n(t+V_i^n(t)) - \hat{G}_i^n(t)\|_T. 
\end{aligned}
\end{equation*}
Since $\hat{A}_i^n$ satisfies \eqref{weak convergence of hat A} and using Corollary \ref{SB for hat G and weak convergence of hat G}, we have the tightness of $\hat{A}_i^n$ and $\hat{G}_i^n$, and they satisfy for any $\epsilon>0$, 
\begin{equation}
\begin{aligned}
    \lim_{\delta\to 0 }\limsup_{n\to\infty} P\left[\omega(\hat{A}_i^n, \delta, T)>\epsilon\right] &= 0, \\
    \lim_{\delta\to 0 }\limsup_{n\to\infty} P\left[\omega(\hat{G}_i^n, \delta, T)>\epsilon\right] &= 0. 
\end{aligned}
\end{equation}
Moreover, we assumed that $\lim_{n\to\infty}\abs{\lambda_i^n/n - \lambda_0} = 0$ by \eqref{regime}. Since $\hat{V}_i^n$ is stochastically bounded and as a consequence, $\|V_i^n\|_T\to 0$ in probability as proved in Proposition \ref{SB for hat V}, above facts imply that $\|\hat{Q}_i^n(t+V_i^n(t)) - \lambda_0\hat{V}_i^n(t)\|_T \to 0$ in probability as $n\to\infty$. 

Now, we are left to show $\|\hat{Q}_i^n(t+V_i^n(t)) - \hat{Q}_i^n(t)\|_T \to 0$ in probability. 
By Theorem \ref{weak convergence theorem}, we have the tightness of $\hat{Q}_i^n$, which also satisfies $\lim_{\delta\to0}\limsup_{n\to\infty} P\left[\omega(\hat{Q}_i^n, \delta, T)>\epsilon\right] = 0 $ for any $\epsilon>0$. Thus, it is straightforward to show the above relation together with the fact that $\|V_i^n\|_T \to 0$ in probability as proved in Proposition \ref{SB for hat V}. This completes the proof. 
\end{proof}


\subsection{Proof of Proposition \ref{SB of hat V (non-abandonment)}}
\label{sec: proof of SB of hat V (non-abandonment)}

\begin{proof}[Proof of Proposition \ref{SB of hat V (non-abandonment)}]
We prove the case for the $i$th queue, and other queues remain identical. 
Here, we first show the stochastic boundedness directly, and then we come back to prove the moment bound condition \eqref{moment bound for hat V and each t (non-abandonment)} for each $t\in[0, T]$ by utilizing the order-preserving property. 

First, we intend to show the stochastic boundedness. For $i$ fixed and let $M>0$ be arbitrary. If $0<M<\hat{V}_i^n(t)$ holds for some $t\in[0, T]$, we know that the queue length at time $t+\frac{M}{\sqrt{n}}$ is not empty, namely $Q_i^n\left(t+\frac{M}{\sqrt{n}}\right)>0$, and satisfies
\begin{equation}
    Q_i^n\left(t+\frac{M}{\sqrt{n}}\right) \geq A_i^n\left(t+\frac{M}{\sqrt{n}}\right) - A_i^n(t).
\end{equation}
With a simple algebraic manipulation by centering and scaling, we obtain a diffusion-scaled inequality
\begin{equation}
    \hat{Q}_i^n\left(t+\frac{M}{\sqrt{n}}\right) \geq \hat{A}_i^n\left(t+\frac{M}{\sqrt{n}}\right) - \hat{A}_i^n(t) + \frac{\lambda_i^n}{n} M.
\end{equation}

Let $0<\delta<1$ and since we assumed $\lambda_i^n/n\to\lambda_0$ as $n\to\infty$, we can find an $\alpha>0$ and $N\geq1$ such that for any $n\geq N$, we have $0<\frac{M}{\sqrt{n}}<\delta$ and $\frac{\lambda_i^n}{n} > 2\alpha > 0$ hold. Therefore, for any $n\geq N$, we have 
\begin{equation*}
\begin{aligned}
    P\left[\|\hat{V}_i^n\|_T > M\right] &\leq P\left[\left\lvert\hat{Q}_i^n\left(t+\frac{M}{\sqrt{n}}\right)\right\rvert + \left\lvert\hat{A}_i^n\left(t+\frac{M}{\sqrt{n}}\right) - \hat{A}_i^n(t)\right\rvert > 2\alpha M\right] \\
    &\leq P\left[\left\lvert\hat{Q}_i^n\left(t+\frac{M}{\sqrt{n}}\right)\right\rvert >\alpha M\right] + P\left[\left\lvert\hat{A}_i^n\left(t+\frac{M}{\sqrt{n}}\right) - \hat{A}_i^n(t)\right\rvert > \alpha M\right] \\
    &\leq P\left[\|\hat{Q}_i^n\|_{T+1} >\alpha M\right] + P\left[\|\hat{A}_i^n(t) - \hat{A}_i^n(s)\|_{0<s<t<(s+\delta)\wedge(T+1)} > \alpha M\right]. 
\end{aligned}
\end{equation*}
The weak convergence of $\hat{A}_i^n$ suggests the tightness of $\hat{A}_i^n$ and the convergence of the modulus of continuity operator of $\hat{A}_i^n$, i.e. $\lim_{\delta\to0}\limsup_{n\to\infty} P\left[\omega(\hat{A}_i^n, \delta, T) > \epsilon\right]=0$. 
Using these facts and the second moment bound condition of $\hat{Q}_i^n$, we have the stochastic boundedness for $\hat{V}_i^n$ and consequently, $\lim_{n\to\infty} \|V_i^n\|_T = 0$ in probability. 

Now, we are left to show the second moment bound result \eqref{moment bound for hat V and each t (non-abandonment)} for each $t\in[0, T]$. 
For each $t\in[0, T]$ fixed and given condition $A_i^n(t) = k$, we have $t_{ik}^n < t < t_{i, k+1}^n$ and $V_i^n(t) = (\max_{j\neq i} \{ t_{j, k+1}^n \} - t)^+$, where $t_{i, k}^n$ represents the arrival time of the $k$th component of category $i$ in the $n$th system. It can be defined as
$t_{i, k}^n = \sum_{j=1}^k \tau_{i, j}^n$, where $\tau_{i, j}^n$'s are inter-arrival times. 
Notice that for a hypothetical component of category $i$ who arrived at time $t$ and $t>t_{j, k+1}^n$ for all $j\neq i$, it needs not to wait and there would be a match immediately since other queues have a component of index $k+1$ waiting to be matched. However, if $t < t_{j, k+1}^n$ for some $j\neq i$, then its waiting time would be their maximum difference $(\max_{j\neq i} \{t_{i, k+1}^n\} - t)$. 
With these facts, we can compute the following conditional moments: 
\begin{equation*}
\begin{aligned}
    E\left[(V_i^n(t))^2 \vert A_i^n(t) = k\right] 
    \leq E\left[\max_{j\neq i}\left\{\left(t_{j, k+1}^n - t\right)^2\right\}\right] 
    \leq \sum_{j\neq i} E\left[\left(t_{j, k+1}^n - t\right)^2\right]. 
\end{aligned}
\end{equation*}
Since we assume renewal arrivals, we let $E[\tau_{jk}^n] = 1/\lambda_j^n$ and $\text{Var}(\tau_{jk}^n) = c_j/(\lambda_j^n)^2$ for some $c_j>0$, and $\{\tau_{jk}^n\}_{k\geq 1}$ are independent with each other, which further yields 
\begin{equation*}
    E[t_{j, k+1}^n] = E\left[\sum_{l=1}^{k+1} \tau_{jl}^n\right] = \frac{k+1}{\lambda_j^n}, \quad
    \text{Var}(t_{j, k+1}^n) = \frac{c_j(k+1)}{(\lambda_j^n)^2}.
\end{equation*}
Therefore, with a simple trick of adding and subtracting $(k+1)/\lambda_j^n$ term, we have
\begin{equation}
    E\left[\left(t_{j, k+1}^n - t\right)^2\right] = E\left[\left(t_{j, k+1}^n - \frac{k+1}{\lambda_j^n} + \frac{k+1}{\lambda_j^n} - t\right)^2\right] 
    \leq 2\left(\frac{c_j(k+1)}{(\lambda_j^n)^2} + \left(\frac{k+1}{\lambda_j^n} - t\right)^2\right). 
    \label{expected value of (t_(j,k+1)^n - t)^2}
\end{equation}
Hence, \eqref{expected value of (t_(j,k+1)^n - t)^2} together with above inequality, we obtain
\begin{equation}
    E\left[(V_i^n(t))^2 \vert A_i^n(t) = k\right] \leq 2\sum_{j\neq i} \left(\frac{c_j(k+1)}{(\lambda_j^n)^2} + \left(\frac{k+1}{\lambda_j^n} - t\right)^2\right). 
\end{equation}
Consequently, we have
\begin{equation*}
\begin{aligned}
    E\left[\abs{\hat{V}_i^n(t)}^2\right] &= \sum_{k=0}^{\infty} E\left[\abs{\hat{V}_i^n(t)}^2 \vert A_i^n(t)=k\right]\cdot P\left[A_i^n(t) = k\right] \\
    &\leq 2n \sum_{k=0}^{\infty} \sum_{j\neq i} \left(\frac{c_j(k+1)}{(\lambda_j^n)^2} + \left(\frac{k+1}{\lambda_j^n} - t\right)^2\right) \cdot P\left[A_i^n(t) = k\right] \\
    &= 2n \sum_{j\neq i}\left(\frac{c_j}{(\lambda_j^n)^2} E\left[A_i^n(t)+1\right] + E\left[\left(\frac{A_i^n(t) + 1}{\lambda_j^n} - t\right)^2\right]\right) \\
    &\leq 2\sum_{j\neq i} \left(c_j\left(\frac{n}{\lambda_j^n}\right)^2  E\left[\Bar{A}_i^n(t) + \frac{1}{n}\right] + \left(\frac{n}{\lambda_j^n}\right)^2  E\left[\left(\hat{A}_i^n(t) + \frac{\lambda_i^n - \lambda_j^n}{\sqrt{n}} t + \frac{1}{\sqrt{n}}\right)^2\right]\right),  
\end{aligned}
\end{equation*}
which further yields \eqref{moment bound for hat V and each t (non-abandonment)}. This completes the proof. 
\end{proof}


\subsection{Proof of Theorem \ref{Theorem convergence of cost functionals (polynomial)}}
\label{sec: proof of Theorem convergence of cost functionals (polynomial)}

\begin{proof}[Proof of Theorem \ref{Theorem convergence of cost functionals (polynomial)}]
Here, we mainly verify the uniform integrability of appropriate integrands by considering the expectations separately under some restrictions for the cost function $\eqref{cost function polynomial growth}$. 

First, we show that
\begin{equation}
    \lim_{n\to\infty} E\left[\sum_{j=1}^K\int_0^{\infty} e^{-\gamma s} C_j(\hat{Q}_j^n(s))ds\right] = E\left[\sum_{j=1}^K \int_0^{\infty} e^{-\gamma s} C_j(X_j(s))ds\right]. 
    \label{convergence of first part}
\end{equation}
Since $\hat{Q}_j^n$ converges weakly to $X_j$ in $D[0, T]$, using the Skorokhod representation theorem, we can simply assume that $\hat{Q}_j^n$ converges to $X_j$ a.s. in some special probability space. By the continuous mapping theorem, we obtain $\lim_{n\to\infty} C_j(\hat{Q}_j^n(t)) = C_j(X_j(t))$ a.s. 

Next, we verify the uniform integrability of the integrand $e^{-\alpha t}C_j(\hat{Q}_j^n(t))$ so that it guarantees the interchange of integral and limit. Since cost function $C_j(\cdot)$ admits polynomial growth as assumed in \eqref{cost function polynomial growth}, we have $C_j(\hat{Q}_j^n(t))\leq c_j(1+\abs{\hat{Q}_j^n(t)}^p)$, where $c_j>0$ and $1\leq p < 2l$ are constants independent of $T$ and $n$ as in \eqref{cost function polynomial growth}. 
Since $1\leq p<2l$ for $l\geq 1$ as assumed, let $\delta>0$ so that $1+\delta = 2l/p$. We will explain the reason for involving $l$ at the end of this section. 
Similar to Proposition \ref{moment bound and stochastical boundedness of hat Q}, we can derive a higher order moment bound for $B_T^n$ random variable introduced in \eqref{assume an upper bound for input}, namely $E\left[B_T^{2l}\right]\leq c(1+T^{d})$.
Following the same proof, we can strengthen the moment-bound condition of the queue lengths by
\begin{equation}
    E\left[\|\vect{\hat{Q}^n}\|_T^{2l}\right] \leq c(1+K)^{2l} (1+T^d) \cdot \exp{\left(2lc_0(1+K) T \right)}, 
    \label{higher order moment bound for hat Q}
\end{equation}
where $d \geq 1$ and $l\geq 1$ are constants independent of $T$ and $n$, and $c>0$ is a genetic constant. 
Notice that if we pick $l=1$, we may obtain a special case proved in \eqref{moment condition on hat Q}. 
This result further renders
\begin{equation*}
    E\left[\abs{\hat{Q}_j^n(s)}^{2l}\right] \leq  E\left[\|\hat{Q}_j^n\|_s^{2l}\right] \leq K^l E\left[\|\vect{\hat{Q}^n}\|_s^{2l}\right] \leq c K^l(1+K)^{2l} (1+s^d) e^{2lc_0(1+K)s}. 
\end{equation*}
Hence, since $\gamma > 2lc_0(1+K)$ as assumed, we obtain
\begin{equation*}
\begin{aligned}
    & E\left[\int_0^{\infty} e^{-\gamma s} \abs{C_j(\hat{Q}_j^n(s))}^{1+\delta}ds\right] \\
    &\leq c\int_0^{\infty} e^{-\gamma s}E\left[\left(1 + \abs{\hat{Q}_j^n(s)}^p\right)^{1+\delta}\right]ds \\
    &\leq c\int_0^{\infty} e^{-\gamma s}\left(1 + E\left[\abs{\hat{Q}_j^n(s)}^{2l}\right]\right)ds \\
    &\leq c\int_0^{\infty} e^{-\gamma s}ds + c K^l (1+K)^{2l}\int_0^{\infty} (1+s^d) e^{-(\gamma - 2lC(1+K))s} ds \\
    &< \infty, 
\end{aligned}
\end{equation*}
where $c>0$ is a generic constant, and $K$, $c_0$, and $d\geq 1$ are constants independent of $n$. 
This verifies the uniform integrability. Therefore, \eqref{convergence of first part} follows. 

Second, we show that
\begin{equation}
    \lim_{n\to\infty} E\left[\sum_{j=1}^K p_j\int_0^{\infty} e^{-\gamma s}d\hat{G}_j^n(s)\right] = E\left[\sum_{j=1}^K p_j \delta_j \int_0^{\infty} e^{-\gamma s} X_j(s) ds\right]. 
    \label{convergence of second part}
\end{equation}
Using Fubini-Tonelli's theorem, we derive that $\gamma\int_{t=0}^{\infty} \int_t^{\infty} e^{-\gamma s}ds d\hat{G}_j^n(t) = \gamma\int_0^{\infty} \int_{t=0}^s e^{-\gamma s}d\hat{G}_j^n(t)ds$, which further implies
\begin{equation}
    \int_0^{\infty} e^{-\gamma t}d\hat{G}_j^n(t) = \gamma\int_0^{\infty}e^{-\gamma t}\hat{G}_j^n(t)dt
    \label{rewritten abandonment cost}
\end{equation}
a.s. 
Notice that this can also be verified using integration by parts and the moment bound of $\hat{G}_k^n$ obtained in Corollary \ref{SB for hat G and weak convergence of hat G}. 
Now, it suffices to show that 
\begin{equation}
     \lim_{n\to\infty} E\left[\sum_{j=1}^K \gamma p_j \int_0^{\infty} e^{-\gamma t}  \hat{G}_j^n(t)dt \right] = E\left[\sum_{j=1}^K  p_j\delta_j\int_0^{\infty} e^{-\gamma t}  X_j(t)dt\right].
     \label{second part converge 2}
\end{equation}

As in the proof of Corollary \ref{SB for hat G and weak convergence of hat G}, since $\|\hat{M}_j^n\|_T$ converges to zero in probability and $\delta_j^n\int_0^{\cdot}\hat{Q}_j^n(s)ds$ converges weakly to $\delta_j\int_0^{\cdot} X_j(s)ds$ in $ D[0, T]$, we conclude that $\hat{G}_j^n(\cdot)$ converges weakly to $\delta_j\int_0^{\cdot} X_j(s)ds$ in $ D[0, T]$. 
Given $\hat{G}_j^n(\cdot) \geq 0$ is non-decreasing, we are left to verify the uniform integrability of $\hat{G}_j^n(T)$ as follows: 
\begin{equation*}
    \begin{aligned}
    E\left[(\hat{G}_j^n(T))^2\right] 
    &\leq 2\left(E\left[\|\hat{M}_j^n\|_T^2\right] + (\delta_j^n)^2 T^2 E\left[\|\hat{Q}_j^n\|_T^2\right]\right) \\
    &\leq 2C_1(1+T^l) + 2C^2 K(1+K)^2 C_2 T^2 (1+T^b) \exp{(2c_0(1+K)T)}, 
    \end{aligned}
\end{equation*}
where $C_1$, $C_2$, $l\geq1$ and $b\geq1$ are constants independent of $T$ and $n$ (see \eqref{moment bound for hat M} and \eqref{moment condition on hat Q}). 
Here the first inequality is obtained by the definition of $\hat{M}_j^n(\cdot)$ introduced in \eqref{hat M martingale}. 
Consequently, $\lim_{n\to\infty}E\left[\hat{G}_j^n(T)\right] = \delta_j E\left[\int_0^TX_j(s)ds\right]$. 
By this limit, the above moment bound condition, and assumption $\gamma>2c_0(1+K)$, we obtain
\begin{equation}
    \lim_{n\to\infty} \gamma\int_0^{\infty} e^{-\gamma t}E\left[\hat{G}_j^n(t)\right]dt = \gamma\int_0^{\infty} e^{-\gamma t} E\left[\int_0^t \delta_j X_j(s)ds\right]dt, 
\end{equation}
by verifying the uniform integrability of integrand, namely
\begin{equation*}
\begin{aligned}
    & E\left[\int_0^{\infty} e^{-\gamma t} \abs{\hat{G}_j^n(t)}^2 dt\right] \\
    &\leq 2\int_0^{\infty} e^{-\gamma t} E\left[\|\hat{M}_j^n\|_t^2\right] dt + 2(\delta_j^n)^2 \int_0^{\infty} e^{-\gamma t} E\left[\|\hat{Q}_j^n\|_t^2\right]t^2 dt \\
    &\leq 2C_1\int_0^{\infty} e^{-\gamma t} (1+t^l) dt + 2C^2 K(1+K)^2 C_2 \int_0^{\infty} t^2 (1+t^b) e^{-(\gamma - 2c_0(1+K)) t} dt \\
    &< \infty, 
\end{aligned}
\end{equation*}
since $\gamma>2c_0(1+K)$ assumed above. 
Using Fubini's theorem, we can rewrite the above conclusion as
\begin{equation}
    \lim_{n\to\infty} \gamma E\left[\int_0^{\infty} e^{-\gamma t}\hat{G}_j^n(t)dt \right]= E\left[\int_0^{\infty} e^{-\gamma t}\delta_j X_j(t)dt\right]. 
\end{equation}
Hence, \eqref{second part converge 2} follows as well as \eqref{convergence of second part}. 
As a consequence, \eqref{convergence of cost functional} immediately follows from \eqref{convergence of first part} and \eqref{convergence of second part}. 
\end{proof}





\end{appendices}


\bibliography{sn-bibliography}


\end{document}